\documentclass[12pt, reqno]{amsproc}
\usepackage{amsmath,amssymb,mathrsfs,hyperref,color}

\DeclareMathAlphabet{\mathpzc}{OT1}{pzc}{m}{it}
\usepackage[nameinlink]{cleveref}

\usepackage{mathtools}
\usepackage[x11names]{xcolor}

\usepackage{paralist}

\usepackage{subfigure}
\usepackage{float}
\usepackage{times}
\usepackage{tikz}
\usepackage{cases}
\usetikzlibrary{intersections}
\usepackage{xcolor}
\usepackage[latin1]{inputenc}
\usepackage{bbm}
\usepackage{tabularx}
\usepackage{booktabs}
\usepackage[labelfont=bf,format=plain,justification=raggedright,singlelinecheck=false]{caption}
\usetikzlibrary{shadings}
\usetikzlibrary[intersections]
\usetikzlibrary{arrows,shapes}

\makeatletter
\def\setliststart#1{\setcounter{\@listctr}{#1}%
  \addtocounter{\@listctr}{-1}}
\makeatother

\makeatletter
\@addtoreset{figure}{section}
\makeatother

\usepackage{calc}

\newtheorem{remark}{\textbf{Remark}}[section]
\newtheorem{example}{\textbf{Example}}[section]
\newtheorem{lemma}{\textbf{Lemma}}[section]
\newtheorem{theorem}{\textbf{Theorem}}[section]
\newtheorem{corollary}{\textbf{Corollary}}[section]
\newtheorem{proposition}{\textbf{Proposition}}[section]

\newtheorem{definition}{\textbf{Definition}}[section]

\numberwithin{equation}{section}

\newcommand{\T}{\mathbb{T}}
\newcommand{\R}{\mathbb{R}}

\newcommand{\N}{\mathbb{N}}

\newcommand{\PP}{\mathcal{P}}
\newcommand{\A}{\mathcal{A}}
\newcommand{\C}{\mathcal{C}}
\newcommand{\M}{\mathcal{M}}
\newcommand{\D}{\mathcal{D}}

\newcommand{\I}{\mathcal{I}}
\newcommand{\W}{\mathcal{W}}

\DeclareMathOperator*{\supp}{spt}

\DeclareMathOperator*{\eps}{\varepsilon}

\DeclareMathOperator*{\PB}{A_{\infty}}
\DeclareMathOperator*{\mane}{\alpha[{\it H}]}
\DeclareMathOperator*{\SR}{SR}
\DeclareMathOperator*{\SC}{SC}

\makeatletter
\def\moverlay{\mathpalette\mov@rlay}
\def\mov@rlay#1#2{\leavevmode\vtop{%
   \baselineskip\z@skip \lineskiplimit-\maxdimen
   \ialign{\hfil$\m@th#1##$\hfil\cr#2\crcr}}}
\newcommand{\charfusion}[3][\mathord]{
    #1{\ifx#1\mathop\vphantom{#2}\fi
        \mathpalette\mov@rlay{#2\cr#3}
      }
    \ifx#1\mathop\expandafter\displaylimits\fi}
\makeatother

\title[Aubry-Mather theory]{Aurby-Mather theory for optimal control systems with nonholonomic constraints}
\author{Piermarco Cannarsa \and Cristian Mendico}
\address{Dipartimento di matematica, Universit\'a degli studi di Roma Tor Vergata -- Via della Ricerca Scientifica 1, 00133 Roma}
\email{cannarsa@mat.uniroma2.it}
\address{Institut de Math\'ematique de Bourgogne - UMR 5584 CNRS, Universit\'e Bourgogne Franche-Comt\'e -- 9 Avenue Alain Savery, 21078 Dijon Cedex, France}
\email{cristian.mendico@u-bourgogne.fr}

\date{\today}

\begin{document}

\begin{abstract}
In this work, we extend Aubry-Mather theory to the case of control systems with nonholonomic constraints. In this framework, we consider an optimal control problem where admissible trajectories are solutions of a control-affine equation. Such an equation is associated with a family of smooth vector fields that satisfy the H\"ormander condition, which implies the controllability of the system. In this case, the Hamiltonian fails to be coercive, so results for Tonelli Hamiltonians cannot be applied. To overcome these obstacles, we develop an intrinsic approach based on the metric properties of the geometry induced on the state space by the sub-Riemannian structure. 
\end{abstract}
\maketitle

\section{Introduction}

The purpose of this paper is to extend Aubry-Mather theory to minimization problems associated with control-affine structure. Our approach is based on  metric properties of the sub-Riemannian geometry induced on the state space by the  dynamical system. In particular, the existence of geodesics w.r.t. the sub-Riemannian metric and controllability properties of the system play a crucial role. We refer the reader to  \cite{bib:ABB, bib:COR, bib:LR} and the references therein for more details on sub-Riemannian geometry and its relation with geometric control.

In \cite{Cannarsa_2022} the authors prove the existence of the critical constant and of a viscosity solution to the ergodic Hamilton-Jacobi equation in a general sub-Riemannian structure. However, in order to improve the natural regularity of critical solutions---which, in general, are merely H\"older continuous w.r.t. the Euclidean distance and Lipschitz continuous w.r.t. the sub-Riemannian one---in this work we restrict the analysis to the class of sub-Riemannian systems that admit no singular minimizing controls different from zero.  We recall that a control function is singular for the problem if there exists a dual arc which is orthogonal to vector fields $\{f_{i}\}_{i=1,\dots, m}$ along the associated trajectory. In absence of such controls, owing to \cite{bib:CR}, we have that critical solutions are semiconcave, hence Lipschitz continuous, w.r.t. the Euclidean distance on $\T^d$. While the hypothesis may seem restrictive, it encompasses many interesting examples, such as the Heisenberg group and the Grushin plane. In \Cref{Grushin1} we investigate in detail the case of the Grushin plane with a Lagrangian function of Ma\~n\'e type.

We observe that a crucial role in the Aubry-Mather theory is played by probability measures on the tangent bundle which are invariant for the Euler-Lagrange flow. Such a strategy fails for nonholonomic control systems because the Euler-Lagrange flow requires distributions to generate the whole tangent space, which is no longer the case. For this reason, we replace such measures by a suitable notion of closed measures (used in \cite{bib:FAS} in order to adapt weak KAM theory to problems in the calculus of variations satisfying mild regularity assumptions), tailored to fit the sub-Riemannian structure. Note that a similar strategy has been used also in the case of control of acceleration (a particular case of single input sub-Riemannian systems with non-zero drift) in \cite{bib:PC1} to define closed measures adapted to the control structure. Furthermore, in this work we generalize the approach in \cite{bib:FAS} by using weaker test functions for closed measures, namely semiconcave functions, which allow us to avoid the problem of the existence of a $C^1$ critical strict subsolution to the ergodic Hamilton-Jacobi equation. The generalization proposed in this work is necessary to show the inclusion of the projected Mather set into the projected Aubry set (see \Cref{inclusion}).  Indeed, the $C^1$ regularity is not natural in our settings due to the weaker controllability of the system. For the existence of smooth subsolutions to the critical Hamilton-Jacobi equation and the use of closed measures, we also refer to \cite{BP2, BP1, BP3}. In \cite{BP3}, P. Bernard developed a Hamiltonian approach to weak KAM theory, that is, an approach where the Euler-Lagrange flow on the tangent bundle is replaced by the Hamiltonian flow on the cotangent bundle. Even in \cite{BP3}, however, the Tonelli property of the Hamiltonian plays a crucial role, which prevents application to control problems with nonholonomic constraints.

Therefore, when the Hamiltonian fails to be Tonelli---as is often the case in optimal control---finding a continuous viscosity solution to the critical equation is a widely open question. Specific models were addressed in \cite{bib:CNS},  \cite{bib:XY},  \cite{bib:PCE},  \cite{bib:BA}, and \cite{bib:ALE}, where the ergodic problem associated with the so-called $G$-equation or other non-coercive Hamiltonians on compact manifolds is treated. See also \cite{bib:ABG, bib:ABE}  \cite{bib:KDV}, and \cite{bib:MG} for more on second order differential games. In the above papers, the authors proved the existence of the critical constant and, in few specific cases such as the $G$-equation, they also showed the existence of a viscosity solution to the critical equation.

 In the case of mechanical Lagrangians and state equations of sub-Riemannian type, the critical constant and a solution of the ergodic equation at the energy level of such a constant were constructed in \cite{Gomes_2007} by using PDE methods. In this work, following a completely different approach that relies on variational methods, we go beyond the existence of the critical constant and critical solutions introducing and studying the Aubry set and the Mather set. In particular, we first show that critical solutions are horizontally differentiable on the projected Aubry set. Then, we investigate dynamical and topological properties of the Mather set. For instance, our novel approach turns out to be crucial to obtain the inclusion of the projected Mather set into the projected Aubry set and to prove the Mather graph property in the case of nonholonomic constraints. We would also like to stress that our results hold for a more general class of Lagrangian functions and control affine structures. Indeed, we allow the Lie algebra generated by the family of vector fields not to have a uniform degree of nonholonomy. Finally, concerning the characterisation of the critical constant in terms of probability measures, in \cite{Gomes_2007}, closed probability measures were defined on the distribution generated by the vector fields associated with the nonholonomic constraints. In this work, we need to change such a definition by introducing probability measures on the control space, which appears to be a natural choice in the sub-Riemannian case. This is why we manage to obtain finer regularity properties for the Aubry and Mather sets (see Section 1.2.1 for more details).

 In \Cref{assumptions} we introduce the notation, the data of the model we are interested in and the main assumptions in force throughout the paper. In \Cref{Aubry}, we introduce Peierls' barrier, whose zero points define the Aubry set. The main result of this chapter concerns the horizontal differentiability of the fixed-point of $T_t$ at any point in the projected Aubry set. We then provide a characterization of the critical constant by minimizing closed measures adapted to the sub-Riemannian structure described in \Cref{Minimizing}. Consequently, we define and study the Mather set in \Cref{Mather_Set}. In particular, we show that the projected Mather set is a subset of the projected Aubry set, and we conclude by proving a weaker form of the Mather graph property.

\section{Assumptions}\label{assumptions}

 Fixed $m \in \N$ such that $m \leq d$, let 
	\begin{equation*}
		f_{i}: \T^{d} \to \R^{d} \quad (i=1, \dots, m)
	\end{equation*}
	and
	\begin{equation*}
		u_{i}:[0,\infty) \to \R \quad (i=1, \dots, m)
	\end{equation*}
be smooth vector fields and square-integrable controls, respectively. Consider the following control affine dynamics
	\begin{align}\label{eq:dynamics}
         \dot\gamma(t)= \displaystyle{\sum_{i=1}^{m}{u_{i}(t) f_{i}(\gamma(t))}}= F(\gamma(t))U(t), \quad t \in [0,+\infty)
	\end{align}
	where $F(x)=[f_{1}(x)| \dots | f_{m}(x)]$ is a $d \times m$ real matrix, $U(t)=(u_{1}(t), \dots, u_{m}(t))^{\star}$ with $(u_{1}, \dots, u_{m})^{\star}$ denoting the transpose of $(u_{1}, \dots, u_{m})$.  
	We assume the vector fields $f_{i}$ to satisfy the following.
\begin{itemize}
	\item[{\bf (F)}] $F = [f_1\; | \ldots |\; f_m] \in C^{1,1}(\T^{d})$ and for any $x \in \T^{d}$ the matrix $F(x)$ has full rank (equal to $m$).
\end{itemize}

 For any $s_{0}$, $s_{1} \in \R$ with $s_{0} < s_{1}$ and $x$, $y \in \T^{d}$ we define
	\begin{align*}\label{eq:controlnotation}
	\begin{split}
	\Gamma_{s_{0}, s_{1}}^{x \to} & =\{(\gamma, u) \in \text{AC}([s_{0}, s_{1}]; \T^{d}) \times L^{2}(s_{0}, s_{1}; \R^{m}): \dot\gamma(t)=F(\gamma(t))u(t),\,\, \gamma(s_{0})=x\},
	\\
	 \Gamma_{s_{0}, s_{1}}^{\to y} & =\{(\gamma, u) \in \text{AC}([s_{0}, s_{1}]; \T^{d}) \times L^{2}(s_{0}, s_{1}; \R^{m}): \dot\gamma(t)=F(\gamma(t))u(t),\,\, \gamma(s_{1})=y\},
	 \\
	 \Gamma_{s_{0}, s_{1}}^{x \to y} & = \Gamma_{s_{0}, s_{1}}^{x \to} \cap \Gamma_{s_{0}, s_{1}}^{\to y}.
	 \end{split}
	\end{align*}
We now state the assumptions on the Lagrangian $L: \T^{d} \times \R^{m} \to \R$.
\begin{itemize}
\item[{\bf (L)}] $L \in C^{2}(\T^{d} \times \R^{m})$ such that $D^{2}_{u} L(x, u) \geq 0$  and there exist $\sigma > 1$, $K_1, K_2 > 0$ such that 
\begin{equation}\label{eq:L0}
L(x, u) \geq K_1|u|^{\sigma} - K_2 \quad \forall\ (x, u) \in \T^d \times \R^m.
\end{equation}
\end{itemize}

In the following we consider the asymptotic behavior of the minimization problem: 
\begin{equation*}
\text{minimize}\quad \int_{0}^{T} L(\gamma(s), u(s))\ ds \quad \text{over all}\quad (\gamma, u) \in \Gamma_{0,T}^{x \to}
\end{equation*}
for any $x \in \T^{d}$ and any time horizon $T > 0$, and we define the value function
\begin{equation}\label{eq:evoValue}
V^{T}(x) = \inf_{(\gamma , u) \in \Gamma_{0,T}^{x \to}} \int_{0}^{T} L(\gamma(s), u(s))\ ds \quad (x \in \T^d).
\end{equation}
Let $H: \T^{d} \times \R^{d} \to \R$ be the Hamiltonian associated with $L$, that is, 
\begin{equation}\label{eq:Hamiltonian}
H(x,p)=\sup_{u \in \R^{m}} \left\{ \displaystyle{\sum_{i=1}^{m}} u_{i} \langle p, f_{i}(x)\rangle  - L(x,u) \right\}, \quad \forall\ (x,p) \in \T^{d} \times \R^{d}.
	\end{equation}
	We need to assume the following. 
\begin{itemize}
	\item[{\bf (S)}] There are no singular minimizing controls of problem \eqref{eq:evoValue} different from zero.
	\end{itemize}
	
	Such extra assumption {\bf (S)} is needed to improve the regularity of viscosity solutions. Indeed, it was proved in \cite[Theorem 1]{bib:CR} that, under {\bf (S)}, the fundamental solution $A_{t}(x,y)$ to the critical equation
	\begin{equation}\label{eq:HJ}
	\mane + H(x, D\chi(x)) = 0, \quad x \in \T^d
	\end{equation}
	is semiconcave (in the Euclidean sense) as a function of $y \mapsto A_{t}(x,y)$ and, consequently, Lipschitz continuous w.r.t. Euclidean distance.
	
	\begin{remark}\em\label{rem:loclipschitz}
In view of the above assumptions on $L$, there exists a constant $C_{H} \geq 0$ such that
	\begin{equation}\label{eq:Hcondition}
	|H(x,p) - H(y,p)| \leq C_{H}(1+|p|^{2})|x-y|, \quad \forall\ x, y \in \T^d.	
	\end{equation}
%
%
\end{remark}

We denote by $L^{*}$ the Legendre transform of $L$, that is,
	\begin{equation*}
	L^{*}(x,p)= \sup_{u \in \R^{m}} \big\{\langle p,v \rangle - L(x,u) \big\},
	\end{equation*}
and we observe that
\begin{equation}\label{eq:starhj}
 H(x, p) = L^{*}(x, F^{*}(x)p), \quad (x, p) \in \T^{d} \times \R^{d}.
\end{equation}

\section{Aubry set}\label{Aubry}

		\begin{definition}[{\bf Dominated functions and calibrated curves}]\label{def:domcal}
			Let $c \in \R$ and let $\varphi$ be a continuous function on $\T^{d}$.			
			\begin{enumerate}
			\item We say that $\varphi$ is dominated by $L-c$, and we denote this by $\varphi \prec L -c$, if for any $a$, $b \in \R$, with $a < b$, and any trajectory-control pair $(\gamma,u) \in \Gamma_{a,b}^{\gamma(a) \to \gamma(b)}$ we have that
			\begin{equation*}
			\varphi(\gamma(b)) - \varphi(\gamma(a)) \leq \int_{a}^{b}{L(\gamma(s), u(s))\ ds} - c\ (b-a).
			\end{equation*}
			\item Fix $a$, $b \in \R$ with $a < b$ and let $(\gamma, u) \in \Gamma_{a,b}^{\gamma(a) \to \gamma(b)}$. We say that $\gamma$ is calibrated for $\varphi$ if
			\begin{equation*}
			\varphi(\gamma(b)) - \varphi(\gamma(a)) = \int_{a}^{b}{L(\gamma(s), u(s))\ ds} - c\ (b-a).
			\end{equation*}
			\end{enumerate}
		\end{definition}
		

		For any $t \geq 0$ and for any $x$, $y \in \T^{d}$ we recall that $A_{t}(x,y)$ stands for the action functional, also called fundamental solution of the critical equation, i.e.,
		\begin{equation}\label{eq:actionfunctional}
		A_{t}(x,y)=\inf_{(\gamma, u) \in \Gamma_{0,t}^{x \to y}} \int_{0}^{t}{L(\gamma(s), u(s))\ ds} .	
		\end{equation}
 We note that $\varphi \prec L-\mane$ if and only if for any $x$, $y$ in $\T^{d}$ and for any $t \geq 0$ we have that 
\begin{equation}\label{eq:subsol}
	\varphi(y)-\varphi(x) \leq A_{t}(x,y) - \mane t.
\end{equation}
Moreover, the following holds. 

\begin{lemma}\label{lem:equilip}
		Assume {\bf (F)}, {\bf (L)} and let $\varphi \prec L-\mane$. Then, $\varphi$ is Lipschitz continuous on $\T^{d}$ with a Lipschitz constant depending only on $L$, i.e., there exists $C_L \geq 0$ such that
		\begin{equation*}
		|\varphi(x) - \varphi(y)| \leq C_L d_{\SR}(x, y) \quad (x, y \in \T^d).
		\end{equation*} 
		\end{lemma}

Peierls' barrier is defined as
\begin{align}\label{eq:PB}
\PB(x,y) =  \liminf_{t \to \infty}\ \big[A_{t}(x,y) - \mane t\big] \quad (x,y \in \T^{d}).	
\end{align}

\begin{lemma}\label{lemma}
For any $x$, $y$, $z \in \T^{d}$ there holds
\begin{equation}\label{eq:peierls}
\PB(x,z) \leq \PB(x,y) + \PB(y,z)
\end{equation}
and, 
\begin{equation}\label{eq:trian}
\PB(x,z) \leq \PB(x,y) + A_{t}(y,z) - \mane t \mbox{ for any } t \geq 0.
\end{equation}
\end{lemma}
%
%

\begin{remark}\label{rem:peierlsubsolution}\em
Note that, combining \eqref{eq:trian} and \eqref{eq:subsol} we have that Peierls' barrier is dominated by $L-\mane$. Moreover, it is well-known that this implies that $\PB(x,\cdot)$ is a subsolution to the critical equation 
\begin{equation*}
\mane + H(x, D \chi(x))=0 \quad (x \in \T^{d}). 
\end{equation*}
\end{remark}

\subsection{Properties of Peierls' barrier}

In this section we derive several properties of  Peierls' barrier. For example, we prove that $\PB(x,\cdot)$ is a fixed-point of the Lax-Oleink and we give a construction of calibrated curves for such a function.

	First, we observe that a fixed-point of the Lax-Oleinik semigroup can be equivalently written as
	\begin{equation}\label{eq:corrector2}
	\chi(x)=\min_{y \in \T^{d}} \{\chi(y) + A_{t}(y,x)\} - \mane t, \quad \forall t > 0, \, \forall x \in \T^{d}.
	\end{equation}
	and it is a solution of 
	\begin{equation}\label{eq:HJcorrector}
	\mane + H(x, D\chi(x))=0, \quad x \in \T^{d}.
	\end{equation}

\begin{lemma}\label{lem:semiconcavity}
Let $\chi$ be a fixed-point of $T_t$. Then, $\chi$ is locally semiconcave. 
\end{lemma}
\proof

Let $y_{t, x} \in \T^{d}$ be the minimum point associated with $(t, x)$ in \eqref{eq:corrector2}. Then, for any $t > 0$, $x$, $h \in \T^{d}$ we have that 
\begin{multline*}
\chi(x+h) + \chi(x-h)- 2\chi(x) 
\\
\leq \chi(y_{t, x}) + \chi(y_{t, x}) - 2\chi(y_{t, x}) + A_{t}(y_{t, x}, x+h) + A_{t}(y_{t, x}, x-h) - 2A_{t}(y_{t, x}, x)
\\
= A_{t}(y_{t, x}, x+h) + A_{t}(y_{t, x}, x-h) - 2A_{t}(y_{t, x}, x).
\end{multline*}
The conclusion follows from the semiconcavity of  $A_{t}(y_{t, x}, \cdot)$, proved in \cite[Theorem 1]{bib:CR}. \qed

The following technical lemma is useful for the analysis of the Aubry set.

\begin{lemma}\label{lem:barrierconvergence}
Let $(x,y) \in \T^{2d}$. Let $\{ t_{n}\}_{n \in \N} \in \R$ and $(\gamma_{n}, u_{n}) \in \Gamma_{0,t_{n}}^{x \to y}$ be such that
\begin{equation}\label{eq:zeros}
t_{n} \to +\infty \quad \text{and} \quad \lim_{n \to +\infty} \int_{0}^{t_{n}}{L(\gamma_{u_{n}}(s), u_{n}(s))\ ds}-\mane t_{n} = \PB(x,y). 	
\end{equation}
Then, there exist a subsequence, still denoted by $(\gamma_{n}, u_{n})$, and a pair $(\bar\gamma, \bar{u}) \in \Gamma_{0,\infty}^{x \to}$ such that
\begin{itemize}
\item[($i$)]  $\{u_{n}\}_{n \in \N}$ weakly converges to $\bar{u}$ in $L^{\sigma}$ on any compact subset of $[0,\infty)$;
\item[($ii$)] $\{\gamma_{n}\}_{n \in \N}$ uniformly converges to $\bar\gamma$ on every compact subset of $[0,\infty)$.
\end{itemize} 
\end{lemma}
\proof
For simplicity of notation we set $h:=\PB(x,y)$. From \eqref{eq:PB} it follows that there exists $\bar{n} \in \N$ such that for any $n \geq \overline{n}$ we have that
	\begin{equation*}
		\int_{0}^{t_{n}}{L(\gamma_{n}(s), u_{n}(s))\ ds}-\mane t_{n} \leq h+1.
	\end{equation*}
 On the other hand, by {\bf (L)} we obtain 
\begin{align*}
\int_{0}^{t_{n}}{L(\gamma_{n}(s), u_{n}(s))\ ds}-\mane t_{n} \geq K_1 \int_{0}^{t_{n}}{|u_{n}(s)|^{\sigma}\ ds}. 
\end{align*}
So,
\begin{equation*}
\int_{0}^{t_{n}}{|u_{n}(s)|^{\sigma}\ ds} \leq \frac{h+1}{K_1}, \quad \forall\ n \geq \overline{n}.	
\end{equation*}
Therefore, there exists a subsequence, still denoted by $\{u_{n}\}$, that weakly converges to an admissible control $\bar{u}$ in $L^{\sigma}$ on any compact subset of $[0,+\infty)$. 
So, for any $s \in [0,t]$ and any $n \geq \overline{n}$
\begin{align*}
\int_{0}^{t_{n}}|\dot\gamma_{n}(s)|^{\sigma}\ ds \leq \int_{0}^{t_{n}}|F(\gamma_{n}(s))|^{\sigma}|u_{n}(s)|^{\sigma}\ ds \leq 4\|F\|^{\sigma}_{\infty}\frac{h+1}{K_1}.	
\end{align*}
Hence, $\{\gamma_{n}\}_{n \in \N}$ is uniformly bounded in $W^{1,\sigma}(0,t; \T^{d})$ for any $t > 0$. Then, by the Ascoli-Arzela Theorem, up to extracting a further subsequence,  $\{\gamma_{n}\}_{n \in \N}$ uniformly converges to a curve $\bar{\gamma}$ on every compact subset of $[0,+\infty)$. 

Now, we claim that $(\bar\gamma, \bar{u})$ satisfies \eqref{eq:dynamics}. Indeed, for any $t \geq 0$ we have that 
\begin{equation*}
\gamma_{n}(t)= x + \sum_{i = 1}^{m}{\int_{0}^{t}{u_{i}^{n}(s)f_{i}(\gamma_{n}(s))\ ds}}. 
\end{equation*}
Thus, by the locally uniform convergence of $\gamma_{n}$ it follows that  $f_{i}(\gamma_{n}(t)) \to f_{i}(\bar\gamma(t))$, locally uniformly, for any $t \geq 0$, as $n \to +\infty$ for any $i = 1, \dots, m$. Therefore, taking $v \in \T^{d}$ we deduce that
\begin{equation*}
\langle v, \gamma_{n}(t) \rangle = \langle v, x \rangle + \sum_{i = 1}^{m}{\int_{0}^{t}{ u_{i}^{n}(s)\langle f_{i}(\gamma_{n}(s)), v \rangle \ ds}}, \quad \forall\ t \geq 0.	
\end{equation*}
As $n \to +\infty$ we get
\begin{equation*}
\langle v, \bar\gamma(t) \rangle = \langle v, x \rangle + \sum_{i = 1}^{m}{\int_{0}^{t}{ \bar{u}_{i}(s) \langle f_{i}(\bar\gamma(s)), v \rangle \ ds}}, \quad \forall\ t \geq 0.
\end{equation*}
Since  $v \in \T^{d}$ is arbitrary, the conclusion follows.
\qed

\begin{remark}\label{rem:conv}\em
Arguing as in the proof of  \Cref{lem:barrierconvergence}, one can prove the following. Given $h \in \R$, $\{ t_{n}\}_{n \in \N}$ and $(\gamma_{n}, u_{n}) \in \Gamma_{-t_{n}, 0}^{x \to y}$ such that
\begin{equation*}
t_{n} \to +\infty \quad \text{and} \quad \lim_{n \to +\infty} \int_{-t_{n}}^{0}{L(\gamma_{u_{n}}(s), u_{n}(s))\ ds}-\mane t_{n} = h, 
\end{equation*}
there exists a subsequence, still denoted by $(\gamma_{n}, u_{n})$, and a trajectory-control pair $(\bar\gamma, \bar{u})$ such that
\begin{itemize}
\item[($i$)]  $\{u_{n}\}_{n \in \N}$ weakly converges to $\bar{u}$ in $L^{\sigma}$ on any compact subset of $(-\infty, 0]$;
\item[($ii$)] $\{\gamma_{n}\}_{n \in \N}$ uniformly converges to $\bar\gamma$ on every compact subset of $(-\infty, 0]$.
\end{itemize} 
Moreover, given $x$, $y \in \T^d$ there exists $\{t_n\}_{n \in \N}$ with $t_n \uparrow \infty$ and $(\gamma_n, u_n) \in \Gamma_{-t_n, 0}^{x \to y}$ such that 
\begin{equation*}
\PB(x, y) = \lim_{n \to \infty}  \int_{-t_n}^{0} L(\gamma_n(s), u_n(s))\ ds - \mane t_n. 
\end{equation*}
\end{remark}

Next we show that there exists a calibrated curve for Peierls' barrier $\PB(x, \cdot)$ at any $x \in \T^{d}$.

\begin{proposition}\label{prop:viscositybarrier}
Let $x \in \T^{d}$. Then, for any $y \in \T^{d}$ there exists $(\bar\gamma, \bar{u}) \in \Gamma_{-\infty,0}^{\to y}$ calibrated for $\PB(x, \cdot)$ in $[-t, 0]$ for any $t > 0$, i.e., $(\bar\gamma, \bar{u})$ satisfies
\begin{equation}\label{eq:calibrata}
	\PB(x,y)-\PB(x,\bar\gamma(-t)) = \int_{-t}^{0}{L(\bar\gamma(s),\bar{u}(s))\ ds} - \mane t, \quad \forall\ t \geq 0.
\end{equation}
\end{proposition}
\proof

Fix $x$, $y \in \T^{d}$  and let $\{ t_{n}\}_{n \in \N}$, $(\gamma_{n}, u_{n}) \in \Gamma_{-t_{n}, 0}^{x \to y}$ be such that
\begin{equation*}
t_{n} \to \infty, \quad \text{and} \quad \lim_{n \to \infty} \int_{-t_{n}}^{0}{L(\gamma_{n}(s), u_{n}(s))\ ds} - \mane t_{n} = \PB(x,y).	
\end{equation*}
Then, by \Cref{rem:conv} there exists $(\bar\gamma, \bar{u})$ such that $u_{n}$ weakly converges to $\bar{u}$ and $\gamma_{n}$ uniformly converges to $\bar\gamma$, on every compact subset of $(-\infty,0]$.

 Fix $t \in [0,\infty)$, take $n \in \N$ such that $d_{n} = d_{\text{SR}}(\bar\gamma(-t), \gamma_{n}(-t)) \leq 1$ and  $t_{n} > t + 1$. Let $(\gamma_{0}, u_{0}) \in \Gamma_{-t, -t+d_{n}}^{\gamma_{n}(-t) \to \bar\gamma(-t)}$ be a geodesic pair and set
\begin{align*}
\widetilde{u}_{n}(s) = 
\begin{cases}
	u_{n}(s), & \quad s \in [-t_{n}, -t]
	\\
	u_{0}(s), & \quad s \in (-t, -t +d_{n}].
\end{cases}	
\end{align*}
We denote by $\widetilde\gamma_{n}$ the associated trajectory, that is, $(\widetilde\gamma_{n}, \widetilde{u}_{n}) \in \Gamma_{-t_{n}, -t+d_{n}}^{x \to \bar\gamma(-t)}$. 


Then, defining the control $\widehat{u}_{n}(s)=\widetilde{u}_{n}(s+t_{n})$, so that $(\widehat\gamma_{n}, \widehat{u}_{n}) \in \Gamma_{0, t_{n}-t+d_{n}}^{x \to \bar\gamma(-t)}$, by {\bf (L)} we get 
\begin{align*}
& A_{t_{n}-t + d_{n}}(x, \bar\gamma(-t)) - \mane(t_{n}-t+d_{n})
\\
\leq\ & \int_{0}^{t_{n} - t + d_{n}}{L(\widehat\gamma_{n}(s), \widehat{u}_{n}(s))\ ds} - \mane (t_{n}-t+d_{n})
\\
=\ & \int_{-t_{n}}^{-t+d_{n}}{L(\widetilde\gamma_{n}(s), \widetilde{u}_{n}(s))\ ds} - \mane (t_{n}-t+d_{n})
\\
=\ & \int_{-t_{n}}^{-t}{L(\gamma_{n}(s), u_{n}(s))\ ds} + \int_{-t}^{-t+d_{n}}{L(\gamma_{0}(s), u_{0}(s))\ ds} - \mane (t_{n}-t+d_{n}).
\end{align*}

Since $\| u_{0}\|_{\infty, [-t, -t+d_{n}]} \leq 1$, by {\bf (L)} we obtain 
\begin{equation*}
\int_{-t}^{-t+d_{n}}{L(\gamma_{0}(s), u_{0}(s))\ ds} \leq 2d_{n}\|L(\cdot, \cdot)\|_{\infty, \T^d \times \overline{B}_1}
\end{equation*}
and so
\begin{align*}
& A_{t_{n}-t + d_{n}}(x, \bar\gamma(-t)) - \mane(t_{n}-t+d_{n})
\\
\leq\ & \int_{-t_{n}}^{-t}{L(\gamma_{n}(s), u_{n}(s))\ ds} + (2\|L(\cdot, \cdot)\|_{\infty, \T^d \times \overline{B}_1}-\mane)d_{n} - \mane(t_{n}-t).
\end{align*}
Hence, by the definition of Peierls barrier we have that 
\begin{equation*}
\PB(x, \bar\gamma(-t)) \leq \liminf_{n \to +\infty} \Big\{A_{t_{n}-t+d_{n}}(x, \bar\gamma(-t)) - \mane (t_{n}-t-d_{n})\Big\}.
\end{equation*}
Appealing to the lower-semicontinuity of the action, we obtain
\begin{equation*}
\int_{-t}^{0}{L(\bar\gamma(s), \bar{u}(s))\ ds} - \mane t \leq \liminf_{n \to +\infty}\left\{\int_{-t}^{0}{L(\gamma_{n}(s), u_{n}(s))\ ds} - \mane t \right\}.	
\end{equation*}
Therefore, combining the above estimates we get 
\begin{multline*}
 \PB(x, \bar\gamma(-t)) + \int_{-t}^{0}{L(\bar\gamma(s), \bar{u}(s))\ ds} - \mane t 
\\
\leq\  \liminf_{n \to +\infty} \Big\{A_{t_{n}-t+d_{n}}(x, \bar\gamma(-t)) - \mane (t_{n}-t-d_{n})\Big\} \\ + \int_{-t}^{0}{L(\bar\gamma(s), \bar{u}(s))\ ds} - \mane t
\\
\leq\  \liminf_{n \to +\infty}\Big\{(2\|L(\cdot, \cdot)\|_{\infty, \T^d \times \overline{B}_1} -\mane)d_{n} \\ + \int_{-t_{n}}^{-t}{L(\gamma_{n}(s), u_{n}(s))\ ds} - \mane(t_{n}-t) \Big\} 
\\
+\  \liminf_{n \to +\infty}\Big\{\int_{-t}^{0}{L(\gamma_{n}(s), u_{n}(s))\ ds} - \mane t \Big\}.	
\end{multline*}
By reordering the terms inside brackets we conclude that
\begin{multline*}
 \PB(x, \bar\gamma(-t)) + \int_{-t}^{0}{L(\bar\gamma(s), \bar{u}(s))\ ds} - \mane t 
\\
\leq \  \liminf_{n \to +\infty}\Big\{(2\|L(\cdot, \cdot)\|_{\infty, \T^d \times \overline{B}_1} - \mane)d_{n} \\ + \int_{-t_{n}}^{0}{L(\gamma_{n}(s), u_{n}(s))\ ds} -\mane t_{n} \Big\} = \PB(x,y).
\end{multline*}
Therefore,
\begin{equation}\label{eq:call}
\PB(x, y) - \PB(x,\bar\gamma(-t)) \geq \int_{-t}^{0}{L(\bar\gamma(s), \bar{u}(s))\ ds} - \mane t	.
\end{equation}

Next, we claim that
\begin{equation}\label{eq:call1}
 \PB(x,y) - \PB(x, \bar\gamma(-t)) \leq \int_{-t}^{0}{L(\bar\gamma(s), \bar{u}(s))\ ds} - \mane t.
\end{equation}
Indeed, by \eqref{eq:trian} we have that 
\begin{equation*}
 \PB(x,y) - \PB(x, \bar\gamma(-t)) \leq A_{t}(\bar\gamma(-t), y) - \mane t.
\end{equation*}
Hence, defining the control 
\begin{equation*}
\widehat{u}(s)=\bar{u}(s-t), \quad s \geq 0
\end{equation*}
and denoting by $\widehat\gamma$ the associated trajectory, we deduce that 
\begin{multline*}
\PB(x,y) - \PB(x, \bar\gamma(-t)) \leq\ A_{t}(\bar\gamma(-t), y) - \mane t
 \\
\leq\ \int_{0}^{t}{L(\widehat\gamma(s), \widehat{u}(s))\ ds} - \mane t  \\ =\ \int_{-t}^{0}{L(\bar\gamma(s), \bar{u}(s))\ ds} - \mane t. 
\end{multline*}
By combining \eqref{eq:call} and \eqref{eq:call1} we obtain \eqref{eq:calibrata}. \qed


Next, we proceed to show that, for any point $x \in \T^d$, the Peierls barrier $\PB(x, \cdot)$ gives a fixed-point the Lax-Oleinik semigroup.

\begin{corollary}\label{cor:viscoHJ}
For each $x \in \T^{d}$ we have that the map $y \mapsto \PB(x,y)$ is a fixed-point of $T_t$, thus a solution of  \eqref{eq:HJcorrector},  on $\T^{d}$. 
\end{corollary}
\proof
Note that $\overline\gamma$, constructed in \Cref{prop:viscositybarrier}, turns out to be calibrated for $\PB(x, \cdot)$ on any interval $[-t,0]$ with $t > 0$ and, thus, by \eqref{eq:trian} we deduce that such a curve is also a minimizer for the action $A_{t}(x, \overline\gamma(-t))$. From \Cref{rem:peierlsubsolution} we have that the map $y \mapsto \PB(x,y)$ is a subsolution to the critical equation for any $x \in \T^{d}$, whereas in view of the fact that $\overline\gamma$ is a minimizing curve we also have that such a map is a supersolution for the critical equation. 

In order to show that $\PB(x, \cdot)$ satisfies \eqref{eq:corrector2}, let $z$ be in $\T^{d}$. Then, \eqref{eq:trian} yields
	\begin{equation*}
	\PB(x,z) \leq \min_{y \in \T^{d}} \{\PB(x,y) + A_{t}(y,z)- \mane t\}.
	\end{equation*}
	Now, by \Cref{prop:viscositybarrier} we have that for any $x$, $z \in \T^{d}$ there exists $(\overline\gamma, \overline{u}) \in \Gamma_{-t,0}^{\to z}$ such that 
	\begin{equation*}
	\PB(x,z)-\PB(x,\bar\gamma(-t)) = \int_{-t}^{0}{L(\bar\gamma(s),\bar{u}(s))\ ds} - \mane t.
	\end{equation*}
	Consequently, there exists $y_{x,z} \in \T^{d}$ such that 
	\begin{equation*}
	\PB(x,z) = \PB(x,y_{x,z}) + A_{t}(y_{x,z},z) - \mane t. 
	\end{equation*}
	This shows that $\PB(x, \cdot)$ satisfies \eqref{eq:corrector2} for any $x \in \T^{d}$.  \qed

\subsection{Projected Aubry set}

In this section, we define the Aubry set and begin by showing some extra properties of  Peierls' barrier on such a set. We conclude by proving that the projected Aubry set $\A$ is a compact subset of $\T^{d}$.

\begin{definition}[{\bf Projected Aubry set}]
		The projected Aubry set $\A$ is defined by
		\begin{equation*}
			\A = \big\{x \in \T^{d} : \PB(x,x)=0 \big\}. 
		\end{equation*}	
		\end{definition}

\begin{theorem}
$\A$ is nonempty. 
\end{theorem}
\proof 

We proceed to show that the $\alpha$-limit set of calibrated curves to $\PB$, which is known to be nonempty from the compactness of the state space, is a subset of $\A$. 
To do so, appealing to Proposition \ref{prop:viscositybarrier}, let $x_0$ and let $y_0 \in \T^d$ and consider $(\gamma, u) \in \Gamma_{-\infty, 0}^{\to y_0}$ calibrated for $\PB(x_0, \cdot)$. For simplicity of notation, we denote by $\Omega$ the $\alpha$-limit set of the calibrated curve $\gamma$. 
Let $x \in \Omega$. Then, there exists $\{t_n\}_{n \in \N}$ with $t_n \uparrow \infty$ such that $\lim_{n \to \infty} \gamma(t_n) = x$. Fix, $\eps > 0$ and $n_1$, $n_2 \in \N$ such that 
\[
|t_{n_1} - t_{n_2}| > \frac{1}{\eps}, \quad \textrm{and}\quad \gamma(-t_{n_1}),\, \gamma(-t_{n_2}) \in B_{\eps}(x)
\]
where $B_{\eps}(x)$ denotes the ball of radius $\eps$ and center $x$ w.r.t. sub-Riemannian distance. Set $d_1 = d_{\SR}(x, \gamma(-t_{n_1}))$, $d_2 = d_{\SR}(x, \gamma(-t_{n_2})$ and let $(\xi_1, u_1) \in \Gamma_{-t_{n_1}+d_1, -t_{n_1}}^{\gamma(-t_{n_1}) \to x}$, $(\xi_2, u_2) \in \Gamma_{-t_{n_2}, -t_{n_2} - d_2}^{x \to \gamma(-t_{n_2})}$ be two geodesic pairs. Define the trajectory-control pair $(\hat\gamma, \hat u) \in \Gamma_{-t_{n_2} - d_2, -t_{n_1} + d_1}^{x \to x}$ as follows
\begin{equation*}
\hat u(s) = 
\begin{cases}
u_2(s), & s \in [-t_{n_2}-d_2, -t_{n_2}]
\\
u(s), & s \in (-t_{n_2}, -t_{n_1}]
\\
u_1(s), & s \in (-t_{n_1}, -t_{n_1} + d_1]. 
\end{cases}
\end{equation*}
%

So, we have 
\begin{align}\label{pp}
\begin{split}
& \inf_{t > \frac{1}{\eps}} [A_t (x, x)  - \mane t]
\\
 \leq\ & A_{d_1}(x, \gamma(-t_{n_1})) + A_{t_{n_2} - t_{n_1}}(\gamma(-t_{n_1}), \gamma(-t_{n_2})) + A_{d_2}(\gamma(-t_{n_{2}}), x)
 \\
 -\ & \mane (t_{n_2} - t_{n_1} + d_1 + d_2). 
 \end{split}
\end{align}
Note that, since $\gamma(-t_{n_1})$, $\gamma(-t_{n_2}) \in B_{\eps}(x)$ and $(\xi_1, u_1)$, $(\xi_2, u_2)$ being geodesics we immediately obtain 
\begin{equation}\label{pp1}
A_{d_i}(x, \gamma(-t_{n_i})) - \mane d_i \leq (\|L(\cdot, \cdot)\|_{\infty, \T^d \times \overline{B}_1} + |\mane|)\eps \quad (i=1, 2). 
\end{equation}
Hence, it remains to estimate $A_{t_{n_2} - t_{n_1}}(\gamma(-t_{n_1}, -t_{n_2}) - \mane (t_{n_2} - t_{n_1})$. Recalling that $(\gamma, u)$ is calibrated for $\PB(x_0, \cdot)$ and that $\PB(\cdot)$ is a fixed-point of $T_t$ we deduce that 
\begin{multline}\label{pp2}
A_{t_{n_2} - t_{n_1}}(\gamma(-t_{n_1}), \gamma(-t_{n_2})) - \mane (t_{n_2} - t_{n_1}) \\ = \PB(x_0, \gamma(-t_{n_1})) - \PB(x_0, \gamma(-t_{n_2})) \leq C_L d_{\SR}(\gamma(-t_{n_1}), \gamma(-t_{n_2})) \leq C_L \eps. 
\end{multline}
Finally, combining \eqref{pp1} and \eqref{pp2} with \eqref{pp} we get
\begin{equation*}
\inf_{t > \frac{1}{\eps}} [A_t(x, x) - \mane t] \leq (C_L + \|L(\cdot, \cdot)\|_{\infty, \T^d \times \overline{B}_1} + |\mane|) \eps
\end{equation*}
that yields to the conclusion taking the supremum over $\eps > 0$ on the right hand side. 
\qed

Next, we provide a characterization of the points in $\A$. 

\begin{theorem}\label{prop:aubrycal}
Let $x \in \T^{d}$. Then, the following properties are equivalent.
\begin{itemize}
\item[($i$)] $x \in \A$. 
\item[($ii$)] There exists a trajectory-control pair $(\gamma_{x}, u_{x}) \in \Gamma_{-\infty, 0}^{\to x} \cap \Gamma_{0,\infty}^{x \to}$ such that  
\begin{equation}\label{eq:rel1}
\PB(\gamma_{x}(t), x) = - \int_{0}^{t}{L(\gamma_{x}(s), u_{x}(s))\ ds} + \mane t	
\end{equation}
and
\begin{equation}\label{eq:rel2}
\PB(x, \gamma_{x}(-t)) = - \int_{-t}^{0}{L(\gamma_{x}(s), u_{x}(s))\ ds} + \mane t.	
\end{equation}
\end{itemize}
Moreover, the trajectory $\gamma_{x}$ in ($ii$) satisfies
\begin{equation}\label{eq:inclusion}
	\gamma_{x}(t) \in \A, \quad t \in \R.
	\end{equation}
\end{theorem}
\proof 

We start by proving that ($i$) implies ($ii$) and, in particular, we begin with \eqref{eq:rel1}. Since $x \in \A$, we have that $\PB(x,x)=0$ . So there exist $\{ t_{n}\}_{n \in \N}$ and $(\gamma^{+}_{n}, u^{+}_{n}) \in  \Gamma_{0,t_{n}}^{x \to x}$ such that
\begin{equation}\label{eq:zeros1}
t_{n} \to +\infty \quad \text{and} \quad \lim_{n \to +\infty} \int_{0}^{t_{n}}{L(\gamma^{+}_{n}(s), u^{+}_{n}(s))\ ds}-\mane t_{n} = 0. 	
\end{equation}
Then, by \Cref{lem:barrierconvergence} there exists $(\gamma^{+}_{x}, u^{+}_{x}) \in \Gamma_{0,\infty}^{x \to}$ such that $u^{+}_{n}$ weakly converges to $u^{+}_{x}$ and $\gamma^{+}_{n}$ uniformly converges to $\gamma^{+}_{x}$, on every compact subset of $[0,\infty)$, respectively.

Fix $t \in [0,+\infty)$ and fix $n$ large enough such that $d_{n} := d_{\SR}(\gamma^{+}_{x}(t), \gamma^{+}_{n}(t)) \leq 1$ and $t +1 < t_{n}$. Let $(\gamma_{0}, u_{0}) \in \Gamma_{t, t+ d_{n}}^{\gamma^{+}_{x}(t) \to\ \gamma^{+}_{n}(t)}$ be a geodesic pair and let $\widetilde{u}_{n} \in L^{2}(t, t_{n}+d_{n})$ be such that 
\begin{align*}
\widetilde{u}_{n}(s) = 
\begin{cases}
	u_{0}(s), & \quad s \in [t, t+d_{n}]
	\\
	u^{+}_{n}(s-d_{n}), & \quad s \in (t+d_{n}, t_{n}+d_{n}],
\end{cases}	
\end{align*}
so that, $(\widetilde\gamma_{n}, \widetilde{u}_{n}) \in \Gamma_{t, t_{n} + d_{n}}^{\gamma^{+}_{x}(t) \to x}$.  
Then, recalling that $\|u_0\|_{\infty, [t, t+d_{n}]} \leq 1$ 
we obtain 
\begin{equation*}
\int_{t}^{t+d_{n}}{L(\gamma_{1}^{n}(s), u_{1}^{n}(s))\ ds} \leq 2d_{n}\|L(\cdot, \cdot)\|_{\infty, \T^d \times \overline{B}_1}.
\end{equation*}
Hence,
\begin{align*}
\int_{t}^{t_{n}}{L(\widetilde\gamma_{n}(s), \widetilde{u}_{n}(s))\ ds} =\ &  \int_{t}^{t+d_{n}}{L(\gamma_{0}(s), u_{0}(s))\ ds} + \int_{t + d_{n}}^{t_{n}+d_{n}}{L(\gamma_{n}(s), u_{n}(s))\ ds} 
\\
\leq\ &   2\|L(\cdot, \cdot)\|_{\infty, \T^d \times \overline{B}_1} d_{n} + \int_{t + d_{n}}^{t_{n}+d_{n}}{L(\gamma_{n}(s), u_{n}(s))\ ds}.
\end{align*}
Now, defining $\widehat{u}_{n}(s)=\widetilde{u}_{n}(s+t)$, that is, $(\widehat\gamma_{n}, \widehat{u}_{n}) \in \Gamma_{0, t_{n}-t}^{ \gamma^{+}_{x}(t) \to x}$, we get
\begin{multline*}
 \PB(\gamma^{+}_{x}(t), x) \leq\  \liminf_{n \to +\infty} \big[A_{t_{n}+d_{n}-t} - \mane(t_{n} + d_{n} -t)\big]
\\
\leq\  \liminf_{n \to +\infty} \left[ \int_{0}^{t_{n} + d_{n} -t}{L(\widehat\gamma_{n}(s), \widehat{u}_{n}(s))\ ds} -\mane(t_{n} + d_{n} -t)\right]
\\
\leq\  \liminf_{n \to +\infty} \Big[\int_{t}^{t+d_{n}}{L(\gamma_{0}(s), u_{0}(s))\ ds} + \int_{t + d_{n}}^{t_{n} + d_{n}}{L(\widetilde\gamma_{n}(s), \widetilde{u}_{n}(s))\ ds} \\ - \mane (t_{n} + d_{n} -t)\Big]
\\
\leq\   \liminf_{n \to +\infty} \Big[ 2\|L(\cdot, \cdot)\|_{\infty, \T^d \times \overline{B}_1}d_{n} + \int_{t + d_{n}}^{t_{n} + d_{n}}{L(\gamma^{+}_{n}(s-d_{n}), u^{+}_{n}(s-d_{n}))\ ds} \\ - \mane (t_{n} + d_{n} -t)\Big]
\\
=\  \liminf_{n \to +\infty} \Big[ \big(2\|L(\cdot, \cdot)\|_{\infty, \T^d \times \overline{B}_1} - \mane\big) d_{n} + \int_{0}^{t_{n}}{L(\gamma^{+}_{n}(s), u^{+}_{n}(s))\ ds}
\\
- \mane t_{n}  - \int_{0}^{t}{L(\gamma^{+}_{n}(s), u^{+}_{n}(s))\ ds} + \mane t \Big].	
\end{multline*}
Then, by \eqref{eq:zeros1}, the uniform convergence of $\gamma_{n}$, the fact that $d_{n} \downarrow 0$ and the lower semicontinuity of the functional we deduce that
\begin{align*}
\int_{0}^{t}{L(\gamma^{+}(s), u^{+}_{x}(s))\ ds} - \mane t + \PB(\gamma^{+}_{x}(t), x) \leq 0. 	
\end{align*}
Moreover, since $\PB(x, \cdot)$ is dominated by $L-\mane$ we also have that
\begin{align*}
\PB(x, \gamma^{+}_{x}(t)) = \PB(x, \gamma^{+}_{x}(t)) - \PB(x,x) \leq \int_{0}^{t}{L(\gamma^{+}_{n}(s), u^{+}_{n}(s))\ ds} - \mane t	
\end{align*}
and by \eqref{eq:peierls} we know that $\PB(x, \gamma^{+}_{x}(t)) + \PB(\gamma^{+}_{x}(t),x ) \geq 0$. Therefore, we obtain
\begin{equation*}
\int_{0}^{t}{L(\gamma^{+}_{x}(s), u^{+}_{x}(s))\ ds} - \mane t +\PB(\gamma^{+}_{x}(t), x) = 0. 	
\end{equation*}
Similar arguments show that there exists $(\gamma^{-}_{x}, u^{-}_{x}) \in \Gamma_{-\infty,0}^{\to x}$ such that \eqref{eq:rel2} holds. 
By taking the sum of \eqref{eq:rel1} and \eqref{eq:rel2} one immediately get that ($ii$) implies ($i$). 
Finally, observe that the inequality 
\begin{equation}\label{eq:claiminclusion}
\PB(\gamma_{x}(t), \gamma_{x}(t)) \leq 0 \quad (t \in\R)
\end{equation}
suffices to show \eqref{eq:inclusion}. Indeed, by \Cref{lemma} we have that $\PB(\gamma_{x}(t), \gamma_{x}(t)) \geq 0$. 
In view of \eqref{eq:peierls}, the following holds 
\begin{equation}\label{eq:triangular}
\PB(\gamma_{x}(t), \gamma_{x}(t)) \leq \PB(\gamma_{x}(t), x) + \PB(x, \gamma_{x}(t)) \quad (t \in \R). 
\end{equation}
Hence, combining \eqref{eq:rel1} and \eqref{eq:rel2} with \eqref{eq:triangular} we get \eqref{eq:claiminclusion}. \qed

We conclude this section showing that the projected Aubry set a compact subset of $\T^d$.

\begin{proposition}\label{prop:closedaubry}
$\A$ is a closed subset of $\T^{d}$.	
\end{proposition}
\proof
Let $\{x_{n}\}_{n \in \N}$ be a sequence in $\A$ such that $\displaystyle{\lim_{n \to \infty}} x_{n}=x \in \T^{d}$. We have to show that $x \in \A$. 

By definition we have that there exist subsequence, still labeled by $\{t_{n}\}_{n \in \N}$, and $\{(\gamma_{k_{n}}, u_{k_{n}})\}_{n \in \N} \in \Gamma_{0,t_{n}}^{x_{n} \to x_{n}}$ such that 
\begin{equation*}
\int_{0}^{t_{n}}{L(\gamma_{k_{n}}(s), u_{k_{n}}(s))\ ds} - \mane t_{n} \leq \frac{1}{n}.
\end{equation*}
Then, by \Cref{lem:barrierconvergence} there exists $(\bar\gamma, \bar{u})$ such that $u_{n}$ weakly converges to $\bar{u}$ and $\gamma_{n}$ uniformly converges to $\bar\gamma$, on every compact subset of $[0,\infty)$, respectively. Let us define $d_{n} = d_{\SR}(x_{n}, x)$ and the control 
\begin{align*}
\widetilde{u}_{n}(s)
\begin{cases}
	u^{n}_{1}(s), & \quad s \in [-d_{n}, 0]
	\\
	u_{n}(s), & \quad s \in (0, t_{n}]
	\\
	u^{n}_{2}(s), & \quad s \in (t_{n}, t_{n} + d_{n}]
\end{cases}	
\end{align*}
where $(\gamma^{n}_{1}, u^{n}_{1}) \in \Gamma_{-d_{n}, 0}^{x \to x_{n}}$ and $(\gamma^{n}_{2}, u^{n}_{2}) \in \Gamma_{t_{n}, t_{n} + d_{n}}^{x_{n} \to x}$ are geodesic pairs, on their respective intervals. Hence, we have that $(\widetilde\gamma_{n}, \widetilde{u}_{n}) \in \Gamma_{-d_{n}, t_{n} + d_{n}}^{x \to x}$.
Let us estimate $|\gamma_{1}^{n}(\cdot)|$ (the same estimate can be obtained for $|\gamma_{2}^{n}(\cdot)|$). Since $\|u^{n}_{1}\|_{\infty, [-d_{n}, 0]} \leq 1$, 
we obtain 
\begin{equation*}
\int_{-d_{n}}^{0}{L(\gamma_{1}^{n}(s), u_{1}^{n}(s))\ ds} \leq d_{n}\|L(\cdot, \cdot)\|_{\infty, \T^d \times \overline{B}_{1}}.
\end{equation*}
Thus, we get
\begin{align*}
& \PB(x,x) \leq \liminf_{n \to \infty}\ [A_{t_{n} + 2d_{n}}(x,x) - \mane(t_{n} + 2d_{n})]
\\
\leq & \liminf_{n \to \infty} \Big( \int_{-d_{n}}^{0}{L(\gamma^{n}_{1}(s), u^{n}_{1}(s))\ ds} + \int_{0}^{t_{n}}{L(\gamma_{n}(s), u_{n}(s))\ ds} -\mane t_{n} 
\\
+ & \int_{t_{n}}^{t_{n} + d_{n}}{L(\gamma^{n}_{2}(s), u^{n}_{2}(s))\ ds} - 2\mane d_{n}\Big) \leq   \lim_{n \to \infty} \left(2d_{n}\|L(\cdot, \cdot)\|_{\infty, \T^d \times \overline{B}_{1}} + \frac{1}{n}\right)= 0.
\end{align*}
The proof is thus complete since, by definition, $\PB(x,x) \geq 0$ for any $x \in \T^{d}$. \qed

	\subsection{Horizontal regularity of critical solutions}
	\label{sec:horizontal}

	In this section we show that any fixed-point of the Lax-Oleinik semigroup is differentiable along the range of $F$ (see the definition below) at any point lying on the projected Aubry set.

\begin{definition}[{\bf Horizontal differentiability}]\label{def:horizontaldifferentiability}
We say that a continuous function $\psi$ on $\T^{d}$ is differentiable at $x \in \T^{d}$ along the range of $F(x)$ (or, horizontally differentiable at $x$) if there exists $q_{x} \in \R^{m}$ such that
\begin{equation}\label{eq:diffF}
\lim_{v \to 0} \frac{\psi(x + F(x)v) - \psi(x) - \langle q_{x}, v \rangle}{|v|} = 0.	
\end{equation}
\end{definition}

	Clearly, if $\psi$ is Frech\'et differentiable at $x$, than $\psi$ is differentiable along the range of $F(x)$ and $q_{x}=F^{\star}(x)D\psi(x)$. For any $\psi \in C(\T^{d})$ we set $D^{+}_{F}\psi(x)=F^{*}(x)D^{+}\psi(x)$.
	
	\begin{lemma}\label{lem:directional}
	Let $\psi \in C(\T^{d})$ be locally semiconcave. Then, $\psi$ is differentiable at $x \in \T^{d}$ along the range of $F(x)$ if and only if $D^{+}_{F}\psi(x) = \{ q_{x}\}$.   	
	\end{lemma}

\noindent Hereafter, the vector $q_{x}$ given in  \Cref{def:horizontaldifferentiability} will be called the {\it horizontal differential} of $\psi$ at $x \in \T^{d}$ and will be denoted by $D_{F}\psi(x)$.

	The next two propositions ensure that any fixed-point of $T_t$, denoted by $\chi$, is differentiable along the range of $F$ at any point lying on a calibrated curve $\gamma$. The proof consists of showing that $D^{+}_{F}\chi$ is a singleton on $\gamma$.  

\begin{proposition}\label{prop:superdiff}
Let $\chi$ be a subsolution to the critical equation and let $(\gamma, u)$ be a trajectory-control pair such that $\gamma: \R \to \T^{d}$ is calibrated for $\chi$. Then we have that
\begin{equation*}
\mane + L^{*}(\gamma(\tau),p)=0, \quad \forall\ p \in D^{+}_{F}\chi(\gamma(\tau))
\end{equation*}
for all $\tau > 0$. 	
\end{proposition}
\proof
On the one hand, since $\chi$ is a subsolution of \eqref{eq:HJcorrector} we have that for any $\tau \geq 0$
\begin{equation*}
\mane + H(\gamma(\tau), p) \leq 0, \quad 	\forall\ p \in D^{+}\chi(\gamma(\tau)). 
\end{equation*}
So, recalling \eqref{eq:starhj} and the fact that $D^{+}_{F}\chi(x)=F^{*}(x)D^{+}\chi(x)$ for any $x \in \T^{d}$ we get
\begin{equation*}
\mane + L^{*}(\gamma(\tau), p) \leq 0, \quad 	\forall\ p \in D^{+}_{F}\chi(\gamma(\tau)) \,\, \forall \tau \geq 0.
\end{equation*} 
Thus, it is enough to prove the reverse inequality. 
	
	Let $\tau > 0$ and $\tau \geq h > 0$. Since $\gamma$ is a calibrated curve for $\chi$ we have that
	\begin{align*}
	& \chi(\gamma(\tau)) - \chi(\gamma(\tau-h)) = \int_{\tau-h}^{\tau}{L(\gamma(s), u(s))\ ds} - \mane h.
	\end{align*}
Then, by the definition of superdifferential we get
\begin{align*}
& \chi(\gamma(\tau)) - \chi(\gamma(\tau-h)) \leq\ \langle p, \gamma(\tau)-\gamma(\tau-h) \rangle + o(h)  
\\
=\ & \left\langle p, \int_{\tau-h}^{\tau}{\dot\gamma(s)\ ds} \right\rangle + o(h)  
 =\ \int_{\tau-h}^{\tau}{ \langle F^{*}(\gamma(s)) p, u(s) \rangle\ ds} + o(h)	
\end{align*}
Therefore,
\begin{align*}
	\int_{\tau-h}^{\tau}{L(\gamma(s), u(s))\ ds} - \mane h \leq \int_{\tau-h}^{\tau}{ \langle F^{*}(\gamma(s)) p, u(s) \rangle\ ds} + o(h)
\end{align*}
or
\begin{align*}
- \mane \leq\ & \frac{1}{h} \int_{\tau-h}^{\tau}{\Big(\langle F^{*}(\gamma(s))p, u(s) \rangle - L(\gamma(s), u(s))\Big)\ ds} + o(1)
\\
\leq\ & \frac{1}{h}\int_{\tau-h}^{\tau}{L^{*}(\gamma(s), F^{*}(\gamma(s))p)\ ds} + o(1).
\end{align*}
Thus, for $h \to 0$ we conclude that
\begin{equation*}
\mane + L^{*}(\gamma(\tau), F^{*}(\gamma(\tau))p) \geq 0. \eqno\square
\end{equation*}

\begin{proposition}\label{prop:differentiability}
Let $\chi$ be a fixed-point of the Lax-Oleinik semigroup and let $(\gamma, u)$ be a trajectory-control pair such that $\gamma: (a,b) \to \T^{d}$ is calibrated for $\chi$, where $-\infty\leq a<b\leq+\infty$. Then, for any $\tau \in(a,b)$ we have that $\chi$ is differentiable at $\gamma(\tau)$ along the range of $F(\gamma(\tau))$. 
\end{proposition}
\begin{proof}
We recall that, owing to \Cref{lem:semiconcavity}, $\chi$ is semiconcave. 
	Moreover, \Cref{prop:superdiff} yields
	\begin{equation*}
	\mane + L^{*}(\gamma(\tau), p) = 0, \quad \forall p \in D^{+}_{F}(\chi(\gamma(\tau)), \,\, \forall \tau \in(a,b).	
	\end{equation*}
Since $L^{*}(x, \cdot)$ is strictly convex and the set $D^{+}_{F}\chi(x)$ is convex,  the above equality implies that $D^{+}_{F}\chi(\gamma(\tau))$ is a singleton. Consequently, \Cref{lem:directional} ensures that  $\chi$ is differentiable at $\gamma(\tau)$ along the range of $F(\gamma(\tau))$. 
\end{proof}

We are now ready to prove the differentiability of any fixed-point of $T_t$ on the Aubry set.

		\begin{theorem}[{\bf Horizontal differentiability on the Aubry set}]\label{thm:aubryset}
		Let $\chi$ be a fixed-point of $T_t$. Then, the following holds. 
		\begin{itemize}
\item[($i$)] For any $x \in \A$ there exists $\gamma_{x}: \R \to \T^{d}$, calibrated for $\chi$, with $\gamma_{x}(0)=x$.
\item[($ii$)] $\chi$ is horizontally differentiable at every $x \in \A$.
		\end{itemize}
\end{theorem}

\begin{proof}

We begin by proving ($i$). Let $x \in \A$ and let $(\gamma_{x}, u_{x}) \in \Gamma_{-\infty,0}^{\to x} \cap \Gamma_{0,\infty}^{x \to}$ be the trajectory-control pair given by \Cref{prop:aubrycal}. We claim that $\gamma_{x}$ is calibrated for $\chi$. By \eqref{eq:corrector2} we have that $\chi \prec L - \mane$ and, for any $t \geq 0$, the following holds
\begin{equation*}
\chi(\gamma_{x}(t)) - \chi(x) \leq \int_{0}^{t}{L(\gamma_{x}(s), u_{x}(s))\ ds} - \mane t.	
\end{equation*}
Moreover, in view of \eqref{eq:corrector2} we deduce that 
\begin{equation*}
	\chi(x) - \chi(\gamma_{x}(t)) \leq A_{s}(\gamma_{x}(t), x) - \mane s 
\end{equation*}
for any $s \geq 0$. Thus, since $\gamma_{x}$ is calibrated for $\PB(\cdot, x)$ we get
\begin{equation*}
\chi(x) - \chi(\gamma_{x}(t)) \leq \PB(\gamma_{x}(t), x) = - \int_{0}^{t}{L(\gamma_{x}(s), u_{x}(s))\ ds} + \mane t. 
\end{equation*}
This proves that $\gamma_{x}$ is a calibrated curve for $\chi$ on $[0,\infty)$. Similarly, one can prove that the same holds on $(-\infty, 0]$. Moreover, for any $s$, $t > 0$ we have that
\begin{multline*}
\chi(\gamma_{x}(t))- \chi(\gamma_{x}(-s)) =\ \chi(\gamma_{x}(t)) - \chi(x) + \chi(x) - \chi(\gamma_{x}(-s))
\\
=\  \int_{0}^{t}{L(\gamma_{x}(\tau), u_{x}(\tau))\ d\tau} - \mane t + \int_{-s}^{0}{L(\gamma_{x}(\tau), u_{x}(\tau))\ d\tau} - \mane s
\\
=\  \int_{-s}^{t}{L(\gamma_{x}(\tau), u_{x}(\tau))\ d\tau} - \mane (t+s),
 \end{multline*}
and this completes the proof of ($i$). 
Conclusion ($ii$) follows from  ($i$) and \Cref{prop:differentiability}. 
\end{proof}

\begin{corollary}\label{cor:feedback}
Let $\chi$ be a fixed-point of $T_t$, let $x \in \A$, and let $\gamma_{x}$ be calibrated for $\chi$ on $\R$ with $\gamma_{x}(0)=x$. Then, $\gamma_{x}$ satisfies the state equation with control 
\begin{equation*}
u_{x}(t)=D_{p}L^{*}(\gamma_{x}(t), D_{F}\chi(\gamma_{x}(t))), \quad \forall t \in\R.
\end{equation*}
Moreover, 
\begin{equation*}
D_{F}\chi(\gamma_{x}(t))=D_{u}L(\gamma_{x}(t), u_{x}(t)), \quad \forall t \in \R. 
\end{equation*}
\end{corollary}
\proof
Let $\chi$ be a fixed-point of $T_t$, let $x \in \A$ and let $\gamma_{x}$ be a calibrated curve for $\chi$. Let $u_{x}$ be the control associated with $\gamma_{x}$. Then, from the maximum principle and the inclusion of the dual arc into the superdifferential of the corresponding value function, e.g. \cite[Theorem 7.4.17]{bib:SC},  it follows that for a.e. $t\in\R$
\begin{equation*}
\langle D_{F}\chi(\gamma_{x}(t)), u_{x}(t) \rangle = L(\gamma_{x}(t), u_{x}(t)) + L^{*}(\gamma_{x}(t), D_{F}\chi(\gamma_{x}(t)).
\end{equation*}
Notice that, $D_{F}\chi(\gamma_{x}(t))$ exists by \Cref{prop:differentiability}.  Hence, by the properties of the Legendre transform and the fact that
$t\mapsto D_{p}L^{*}(\gamma_{x}(t), D_{F}\chi(\gamma_{x}(t)))$ is continuous on $\R$, 
we conclude that $u_{x}$ has a continuous extension to $\R$ given by 
$
u_{x}(t)=D_{p}L^{*}(\gamma_{x}(t), D_{F}\chi(\gamma_{x}(t)))
$.
\qed

\begin{remark}\em
Following the classical Aubry-Mather theory for Tonelli Hamiltonian systems, one could  define the Aubry set $\widetilde\A \subset \T^{d} \times \R^{m}$ as 
\begin{equation*}
\widetilde\A = \bigcap\{(x,u) \in \A \times \R^{m}: D_{F}\chi(x)=D_{u}L(x,u)\}
\end{equation*}
where the intersection is taken over all fixed-point $\chi$ of the Lax-Oleinik semigroup. 
\end{remark}

\section{Minimizing measures and critical constant}\label{Minimizing}

\subsection{Horizontally closed measures}
 We begin by introducing a class of probability measures that adapts the notion of closed measures to sub-Riemannian control systems.
Set
\begin{equation*}
\PP^{\sigma}(\T^{d} \times \R^{m}) = \left\{\mu \in \PP(\T^{d} \times \R^{m}) : \int_{\T^{d} \times \R^{m}}{|u|^{\sigma}\ \mu(dx,du)} < \infty \right\}
\end{equation*}
and for $\kappa > 0$ let $\PP^{\sigma}_{\kappa}(\T^d \times \R^m)$ be the subset of $\PP^\sigma(\T^d \times \R^m)$ such that $\mu \in \PP^{\sigma}_{\kappa}(\T^d \times \R^m)$ if 
\begin{equation*}
 \int_{\T^{d} \times \R^{m}}{|u|^{\sigma}\ \mu(dx,du)} \leq \kappa.
\end{equation*}
Note that, since $\PP_{\kappa}^{\sigma}(\T^{d} \times \R^{m})$ is closed w.r.t. distance $d_{1}$ by Prokhorov's Theorem $\PP_{\kappa}^{\sigma}(\T^{d} \times \R^{m})$ is compact w.r.t. distance $d_{1}$ for any $\kappa > 0$. 


\begin{definition}[{\bf $F$-closed measure}]\label{def:Closedmeasure}
We say that $\mu \in \PP^{\sigma}(\T^{d} \times \R^{m})$ is an $F$-closed measure if
	\begin{equation*}
	\int_{\T^{d} \times \R^{m}}{\langle F^{\star}(x)D\varphi(x), u \rangle\ \mu(dx,du)}=0, \quad \forall\ \varphi \in C^{1}(\T^{d}). 	
	\end{equation*}
	We denote by $\C_{F}$ the set of all $F$-closed measures and we set $\C^{\kappa}_{F} = \C_F \cap \PP_{\kappa}^{\sigma}(\T^d \times \R^m)$.
\end{definition}

 Closed measures were first introduced in \cite{Mane1, Mane2}  and then used, for instance, in \cite{bib:FAS} in order to overcome the lack of regularity of the Lagrangian. Indeed, if $L$ is merely continuous, then there is no Euler-Lagrange flow and, consequently, it is not possible to introduce invariant measures as it is customary, see for instance \cite{bib:FA}. Similarly, in our setting, such a flow does not exist: for this reason, the use of closed measures turns out to be necessary. Moreover, as we will show in the next result, such measures collect the behavior of the minimizing trajectories for the infimum in \eqref{eq:evoValue} as $T \to \infty$.

We now proceed to construct one closed measure that will be particularly useful to study the Aubry set. Given $x_{0} \in \T^{d}$, for any $T > 0$ let the pair $(\gamma_{x_{0}}, u_{x_{0}}) \in \Gamma_{0,T}^{x_{0} \to}$ be optimal for \eqref{eq:evoValue}. Define the probability measure $\mu^{T}_{x_{0}}$ by 
\begin{equation}\label{eq:uniformmeasure}
\int_{\T^{d} \times \R^{m}}{\varphi(x,u)\ \mu^{T}_{x_{0}}(dx,du)}= \frac{1}{T}\int_{0}^{T}{\varphi(\gamma_{x_{0}}(t), u_{x_{0}}(t))\ dt}, \quad \forall\ \varphi \in C_{b}(\T^{d} \times \R^{m}).
\end{equation}
Then, we have the following.  
 
\begin{proposition}\label{prop:nonempty}
The sequence $\{ \mu^{T}_{x_{0}}\}_{T>0}$ is tight and there exists a sequence $T_{n} \to \infty$ such that $\mu^{T_{n}}_{x_{0}}$ weakly-$^*$ converges to an $F$-closed measure $\mu^{\infty}_{x_{0}}$.  
\end{proposition}

\proof First, by definition it follows that $\{ \pi_{1} \sharp \mu^{T}_{x_{0}}\}_{T >0}$ has compact support, uniformly in $T$. Thus, such a family of measures is tight. Let us prove that $\{ \pi_{2} \sharp \mu^{T}\}_{T > 0}$ is also tight.	

On the one hand, by the upper bound in \eqref{eq:L0} we have that
	\begin{equation*}
	\frac{1}{T}V_{T}(x_{0}) \leq \frac{1}{T}\int_{0}^{T}{L(x_{0}, 0)\ ds} \leq  \| L(\cdot, 0)\|_{\infty}.
	\end{equation*}
On the other hand, since $(\gamma_{x_{0}}, u_{x_{0}})$ is a minimizing pair for $V_{T}(x_{0})$, by the lower bound in \eqref{eq:L0} we get
\begin{align*}
& \frac{1}{T}V_{T}(x_{0}) = \frac{1}{T}\int_{0}^{T}{L(\gamma_{x_{0}}(t), u_{x_{0}}(t))\ dt} 
\\
= &  \int_{\T^{d} \times \R^{m}}{L(x,u)\ \mu^{T}_{x_{0}}(dx,du)} \geq \int_{\T^{d} \times \R^{m}}{\big(K_1|u|^{\sigma} - K_2\big)\ \mu^{T}_{x_{0}}(dx,du)}. 	
\end{align*}
Therefore,
\begin{equation*}
K_1\int_{\T^{d} \times \R^{m}}{|u|^{\sigma}\ \mu^{T}_{x_{0}}(dx,du)} \leq   \| L(\cdot, 0)\|_{\infty} + K_2. 
\end{equation*}
Consequently, the family of probability measures $\{ \pi_{2} \sharp \mu^{T}_{x_{0}}\}_{T >0}$ has bounded (w.r.t. $T$) $\sigma$-moment. So,  $\{ \pi_{2} \sharp \mu^{T}_{x_{0}}\}_{T >0}$ is tight.

Since $\{\pi_{1} \sharp \mu^{T}_{x_{0}}\}_{T > 0}$ and $\{\pi_{2} \sharp \mu^{T}_{x_{0}}\}_{T > 0}$ are tight, so is $\{\mu^{T}_{x_{0}}\}_{T >0}$ by \cite[Theorem 5.2.2]{bib:AGS}.  Therefore, by Prokhorov's Theorem there exists  $\{T_{n}\}_{n \in \N}$, with $T_{n} \to \infty$, and $\mu^{\infty}_{x_{0}} \in \PP^{\sigma}(\T^{d} \times \R^{m})$ such that $\mu^{T_{n}}_{x_{0}} \rightharpoonup^{*} \mu^{\infty}_{x_{0}}$.

We now show that $\mu^{\infty}_{x_{0}}$ is an $F$-closed measure, that is,
\begin{equation*}
	\int_{\T^{d} \times \R^{m}}{\langle F^{\star}(x)D\psi(x), u \rangle\ \mu^{\infty}_{x_{0}}(dx,du)}=0 \quad (\psi \in C^{1}(\T^{d})).
\end{equation*}
By definition we have that 
\begin{align*}
& \int_{\T^{d} \times \R^{m}}{\langle F^{\star}(x)D\psi(x), u\rangle \mu_{x_{0}}^{T_{n}}(dx,du)} = \frac{1}{T_{n}}\int_{0}^{T_{n}}{\langle F^{\star}(\gamma_{x_{0}}(t)D\psi(\gamma_{x_{0}}(t)), u_{x_{0}}(t)\rangle\ dt}
\\
=\ & \frac{1}{T_{n}}\int_{0}^{T_{n}}{\langle D\psi(\gamma_{x_{0}}(t), \dot\gamma_{x_{0}}(t)\rangle\ dt} = \frac{\psi(\gamma_{x_{0}}(T_{n})) - \psi(x_{0})}{T_{n}}. 
\end{align*}
So, in view of the compactness of the state space we obtain
\begin{equation*}
\int_{\T^{d} \times \R^{m}}{\langle F^{\star}(x)D\psi(x), u \rangle\ \mu^{\infty}_{x_{0}}(dx,du)} = \lim_{n\to \infty} \frac{\psi(\gamma_{x_{0}}(T_{n})) - \psi(x_{0})}{T_{n}} = 0. \eqno\square
\end{equation*}

	The following property, which is interesting in its own right, will be crucial for the characterization of the critical constant derived in \Cref{thm:main1} below.

	\begin{proposition}\label{prop:minimax}
	We have that
	\begin{equation}\label{eq:RR}
	\inf_{\mu \in \C_{F}} \int_{\T^{d} \times \R^{m}}{L(x,u)\ \mu(dx,du)}=-\inf_{\psi \in C^{1}(\T^{d})} \sup_{x \in \T^{d}} H(x, D\psi(x)). 	
	\end{equation}
	\end{proposition}
	
	The following two lemmas are needed for the proof of \Cref{prop:minimax}.

	\begin{lemma}\label{lem:minmaxprel}
		We have
		\begin{align}\label{eq:hand1}
	\begin{split}
	& \inf_{\mu \in \C_{F}} \int_{\T^{d} \times \R^{m}}{L(x,u)\ \mu(dx,du)}	
	\\
	= &  \inf_{\mu \in \PP^{2}(\T^{d} \times \R^{m})}\sup_{\psi \in C^{1}(\T^{d})}\int_{\T^{d} \times \R^{m}}{\Big(L(x,u) - \langle F^{\star}(x)D\psi(x),u \rangle  \Big)\ \mu(dx,du)}.
	\end{split}
	\end{align}
   \end{lemma}

The proof of the above lemma is based on an argument which is quite common in optimal transport theory (see, e.g., \cite[Theorem 1.3]{bib:CV}). We give the reasoning for the reader's convenience. 


 \begin{lemma}\label{lem:weakconv}
Let $\phi \in C(\T^{d} \times \R^{m})$ be such that $$\phi_{0} \leq \phi(x,u) \leq C_{\phi}(1+|u|^{2}), \quad \forall\ (x,u) \in \T^{d} \times \R^{m}$$
for some constants $\phi_{0} \in \R$ and $C_{\phi} \geq 0$. Let $\{\mu_{j}\}_{j \in \N} \in \PP^{2}(\T^{d} \times \R^{m})$ and let $\mu \in \PP^{2}(\T^{d} \times \R^{m})$ be such that $\mu_{j} \rightharpoonup^{*} \mu$ as $j \to \infty$. Then, we have that
\begin{equation}\label{eq:liminfineq}
	\liminf_{j \to \infty} \int_{\T^{d} \times \R^{m}}{\phi(x,u)\ \mu_{j}(dx,du)} \geq \int_{\T^{d} \times \R^{m}}{\phi(x,u)\ \mu(dx,du)}. 
\end{equation}
\end{lemma}

\noindent{\it Proof of \Cref{prop:minimax}.}
We divide the proof into two steps.

\noindent {\bf (1).}  We will prove that 
	\begin{align}\label{eq:claimminimax}
	\begin{split}
		&  \inf_{\mu \in \PP^{\sigma}(\T^{d} \times \R^{m})}\sup_{\psi \in C^{1}(\T^{d})}\int_{\T^{d} \times \R^{m}}{\Big(L(x,u)- \langle F^{\star}(x)D\psi(x),u \rangle  \Big)\ \mu(dx,du)}
		\\
		= &  \sup_{\psi \in C^{1}(\T^{d})} \inf_{\mu \in \PP^{\sigma}(\T^{d} \times \R^{m})} \int_{\T^{d} \times \R^{m}}{\Big(L(x,u)- \langle F^{\star}(x)D\psi(x),u \rangle  \Big)\ \mu(dx,du)}. 
		\end{split}
	\end{align}
For the proof of \eqref{eq:claimminimax} we will apply the Minimax Theorem (\cite[Theorem A.1]{bib:OPA}) to the functional 
\[
\mathcal{F}: C^{1}(\T^{d}) \times \PP^{\sigma}(\T^{d} \times \R^{m}) \to \R
\]
defined by
	\begin{equation*}
	\mathcal{F}(\psi, \mu) = \int_{\T^{d} \times \R^{m}}{\Big(L(x,u)-\langle F^{\star}(x)D\psi(x), u \rangle \Big)\ \mu(dx,du)}.
	\end{equation*}
	To do so, define
	\begin{align*}
	c^{*}=2\|L(\cdot, 0)\|_{\infty}. 	
	\end{align*}
In order to check that the assumptions of the Minimax Theorem are satisfied we have to prove the following.
 \begin{itemize}
 \item[($i$)] $\mathcal{E}:=\big\{\mu \in \PP^{\sigma}(\T^{d} \times \R^{m}):\ \mathcal{F}(0,\mu) \leq c^{*} \big\}$ is compact in  $\big(\PP^{\sigma}(\T^{d} \times \R^{m}), d_{1}\big)$.   Indeed, for any given $\mu \in \PP^{\sigma}(\T^{d} \times \R^{m})$ we know that the support of $\pi_{1} \sharp \mu$ is contained in $\T^{d}$. Moreover, the assumption {\bf (L)} implies that for any $\mu \in \mathcal{E}$ we have that the family $\{\pi_{2} \sharp \mu\}_{\mu \in \mathcal{E}}$ has bounded $\sigma$ moment and, therefore, it is tight. Thus, $\mu$ is tight by \cite[Theorem 5.2.2]{bib:AGS} and  $\mathcal{E}$ is compact by Prokhorov's Theorem.
 \item[($ii$)] The map $\mu \mapsto \mathcal{F}(\psi, \mu)$ is lower semicontinuous on $\mathcal{E}$ for any $\psi \in C^{1}(\T^{d})$. This fact follows from Lemma \ref{lem:weakconv}.
 \end{itemize}

Therefore, by the Minimax Theorem we  obtain \eqref{eq:claimminimax}.

\noindent{\bf (2).} In order to complete the proof we observe that, in view of \eqref{eq:hand1} and \eqref{eq:claimminimax}, 
 \begin{align*}
 	& \inf_{\mu \in \C_{F}} \int_{\T^{d} \times \R^{m}}{L(x,u)\ \mu(dx,du)}
 	\\
 	= & \inf_{\mu \in \PP^{\sigma}(\T^{d} \times \R^{m})}\sup_{\psi \in C^{1}(\T^{d})}\int_{\T^{d} \times \R^{m}}{\Big(L(x,u)- \langle F^{\star}(x)D\psi(x),u \rangle  \Big)\ \mu(dx,du)}
 	\\
 	= &  \sup_{\psi \in C^{1}(\T^{d})} \inf_{\mu \in \PP^{\sigma}(\T^{d} \times \R^{m})} \int_{\T^{d} \times \R^{m}}{\Big(L(x,u)- \langle F^{\star}(x)D\psi(x),u \rangle  \Big)\ \mu(dx,du)}.
\end{align*}
Now, assumption {\bf (L)} ensures the existence of 
\begin{equation*}
 \min_{(x,u) \in \T^{d} \times \R^{m}}\Big\{ L(x,u)- \langle F^{\star}(x)D\psi(x),u \rangle \Big\}.
\end{equation*}
Therefore, by taking a Dirac mass centered at any pair $(\overline{x}, \overline{u})$ at which the above minimum is attained, one deduces that
\begin{align*}
&  \sup_{\psi \in C^{1}(\T^{d})} \inf_{\mu \in \PP^{\sigma}(\T^{d} \times \R^{m})} \int_{\T^{d} \times \R^{m}}{\Big(L(x,u)- \langle F^{\star}(x)D\psi(x),u \rangle  \Big)\ \mu(dx,du)}
\\
 =\	& \sup_{\psi \in C^{1}(\T^{d})} \min_{(x,u) \in \T^{d} \times \R^{m}}\Big\{ L(x,u)- \langle F^{\star}(x)D\psi(x),u \rangle \Big\}.	
\\
	 	=\ & \sup_{\psi \in C^{1}(\T^{d})} \left(-\max_{(x,u) \in \T^{d} \times \R^{m}}\Big\{ \langle F^{\star}(x)D\psi(x),u \rangle - L(x,u)  \Big\}\right)
	\\
	=\ & - \inf_{\psi \in C^{1}(\T^{d})} \max_{(x,u) \in \T^{d} \times \R^{m}}\Big\{  \langle F^{\star}(x)D\psi(x),u \rangle - L(x,u) \Big\}
	\\
 	= &   - \inf_{\psi \in C^{1}(\T^{d})} \max_{ x \in \T^{d}}\  H(x, D\psi(x)).
 \end{align*}
We observe that the last inequality holds because 
\begin{align*}
& \max_{(x,u) \in \T^{d} \times \R^{m}}\Big\{\langle F^{\star}(x)D\psi(x), u \rangle - L(x,u) \Big\}
\\
=\ & \max_{x \in \T^{d}} \sup_{u \in \R^{m}}\Big\{\langle F^{\star}(x)D\psi(x), u \rangle - L(x,u) \Big\} =  \max_{x \in \T^{d}} H(x, D\psi(x)). 
\end{align*}
This completes the proof. \qed

\begin{theorem}\label{thm:main1}
The following holds:
\begin{equation}\label{eq:mumane}
\mane  = \min_{\mu \in \C_{F}} \int_{\T^{d} \times \R^{m}}{L(x,u)\ \mu(dx,du)}. 
\end{equation}
\end{theorem}

Once again, we need a lemma for the proof of \Cref{thm:main1}. 

\begin{lemma}\label{lem:smoothapprox}
Let $\chi$ be a solution of the critical equation. Then, there exists a constant $K \geq 0$ such that, for any $\eps > 0$
\begin{equation}\label{eq:smoothsub}
	\mane + H(x, D\chi_{\eps}(x)) \leq K\eps \quad \forall\ x \in \T^{d},
\end{equation}
where $\chi_{\eps}(x)= \chi \star \xi_{\eps}(x)$ and $\xi_{\eps}$ is a smooth mollifier. 
\end{lemma}
%
%

\noindent {\it Proof of \Cref{thm:main1}.} Observe that it is enough to show that there exists $\kappa_0 \geq 0$ such that for any $\kappa \geq \kappa_0$ we have 
\begin{equation*}
\mane = \inf_{\mu \in \C^{\kappa}_{F}} \int_{\T^d \times \R^m} L(x, u) \mu(dx, du).
\end{equation*}

First, we recall that by \cite[Theorem 5.3]{Cannarsa_2022} we have that 
 \begin{align*}
 \mane = & \lim_{T \to +\infty} \frac{1}{T} v^{T}(0).  
 \end{align*}
 Hence, appealing to \Cref{lem:weakconv} and recalling that $L(x,u) \geq -K_2$ we obtain
 \begin{align}\label{eq:star}
 \begin{split}
  \mane = & \lim_{T \to \infty}\int_{\T^{d} \times \R^{m}}{L(x,u)\ \mu^{T}_{0}(dx,du)} \geq  \int_{\T^{d} \times \R^{m}}{L(x,u)\ \mu^{\infty}_{0}(dx,du)}.
 \end{split} 
 \end{align}
Here, we recall that $\mu^{T}_{0}$ is defined by
\begin{equation*}
\int_{\T^{d} \times \R^{m}}{\varphi(x,u)\ \mu^{T}_{0}(dx,du)}= \frac{1}{T}\int_{0}^{T}{\varphi(\gamma_{0}(t), u_{0}(t))\ dt}, \quad \forall\ \varphi \in C_{b}(\T^{d} \times \R^{m})
\end{equation*}
with $(\gamma_{0}, u_{0}) \in \Gamma_{0,T}^{0 \to}$ optimal for \eqref{eq:evoValue} and $\mu^{\infty}_{0}$ denotes the limit measure of $\{\mu^{T}_{0}\}_{T >0}$ constructed in \Cref{prop:nonempty}. So, setting
\begin{equation*}
\kappa_0 = \frac{ \| L(\cdot, 0)\|_{\infty} + K_2}{K_1}
\end{equation*}
we have $\mu_{0}^{\infty} \in \C^{\kappa}_{F}$ for any $\kappa \geq \kappa_0 $, which implies
\begin{equation*}
\mane \geq \min_{\mu \in \C^{\kappa}_{F}}\int_{\T^{d} \times \R^{m}}{L(x,u)\ \mu(dx,du)} \quad (\forall\ \kappa \geq \kappa_0).	
\end{equation*}
Next, by \Cref{prop:minimax} we have that for any $\psi \in C^{1}(\T^{d})$
\begin{align}\label{eq:ineqq}
\begin{split}
\inf_{\mu \in \C_{F}} \int_{\T^{d} \times \R^{m}}{L(x,u)\ \mu(dx,du)} = -\inf_{\psi \in C^{1}(\T^{d})} \sup_{x \in \T^{d}} H(x, D\psi(x)). 
\end{split}
\end{align}
Let $\chi$ be a fixed-point of $T_t$. For $\eps \geq 0$ let $\chi_{\eps}(x)= \chi \star \xi^{\eps}(x)$, where $\xi^{\eps}$ is a smooth mollifier. So, \Cref{lem:smoothapprox} yields 
\begin{equation*}
\mane + H(x, D\chi_{\eps}(x)) \leq K\eps, \quad x \in \T^d. 	
\end{equation*}
Then, using $\chi_{\eps}$ to give a lower bound for the right hand side of \eqref{eq:ineqq} we obtain
\begin{align*}
\inf_{\mu \in \C_{F}} \int_{\T^{d} \times \R^{m}}{L(x,u)\ \mu(dx,du)}
\geq\  - \sup_{x \in \T^{d}} H(x, D\chi_{\eps}(x)) \geq \mane - K\eps.
\end{align*}
Hence,
\begin{multline*}
 \min_{\mu \in \C^{\kappa}_{F}} \int_{\T^{d} \times \R^{m}}{L(x,u)\ \mu(dx,du)} \geq \inf_{\mu \in \C_{F}} \int_{\T^{d} \times \R^{m}}{L(x,u)\ \mu(dx,du)}
\\
\geq\ \mane - K\eps \quad (\forall\ \kappa \geq \kappa_0)
\end{multline*} 
as $\eps \downarrow 0$ we get 
\begin{equation*}
\min_{\mu \in \C^{\kappa}_{F}} \int_{\T^{d} \times \R^{m}}{L(x,u)\ \mu(dx,du)} \geq \mane 	
\end{equation*}
for any $\kappa \geq \kappa_0$ and this completes the proof.\qed

\begin{corollary}\label{cor:carmin}
For any $x_{0} \in \T^{d}$ we have that
\begin{equation*}
\mane = \int_{\T^{d} \times \R^{m}}{L(x,u)\ \mu^{\infty}_{x_{0}}(dx,du)}	
\end{equation*}
where $\mu^{\infty}_{x_{0}}$ is given in \Cref{prop:nonempty}.
\end{corollary}

%

\subsection{Strongly closed measures}

Here, we restrict the notion of $F$-closed measures in order to allow test functions to be semiconcave, which is, a larger class of test functions than the space $C^1$ considered in \Cref{def:Closedmeasure}.

For brevity of notation we denote by $\SC(\T^{d})$ the set of all semiconcave function defined on $\T^{d}$. Moreover, we introduce the one-sided horizontal directional derivative 
\begin{equation*}
\partial_{F} \varphi(x, \theta) = \lim_{h \to 0^{+}} \frac{\varphi(x + hF(x)\theta) - \varphi(x)}{h}, \quad \varphi \in \SC(\T^{d}) 
\end{equation*}
 at any point $x \in \T^{d}$ and direction $\theta \in \R^{m}$. 
 
 \begin{definition}
 We say that $\mu \in \PP^{\sigma}(\T^{d} \times \R^{m})$ is a strongly $F$-closed measure if
	\begin{equation*}
	\int_{\T^{d} \times \R^{m}}{\partial_{F}\varphi(x, u)\ \mu(dx,du)}=0, \quad \forall\ \varphi \in \SC(\T^{d}). 	
	\end{equation*}
	We denote by $\C^{S}_{F}$ the set of all $F$-closed measures and by $\C^{S}_{F}(\kappa)= \C^{S}_{F} \cap \PP^{\sigma}_{\kappa}(\T^d \times \R^m)$. 
 \end{definition}

\begin{lemma}
$\C^{S}_{F}(\kappa)$ is a compact subset of $\PP^{\sigma}_{\kappa}(\T^{d} \times \R^{m})$ w.r.t. distance $d_1$.
\end{lemma}
\proof
Let $\varphi \in \SC(\T^{d})$, let $\{\mu_{k}\}_{k \in \N} \subset \C^{S}_{F}$ be such that $\mu_{k} \rightharpoonup^{*} \mu$ and let $x \in \T^{d}$, $h \in \T^{d}$ $u \in \R^{m}$. Let $\omega$ be the modulus of semiconcavity of $\varphi$. Then, we have
\begin{equation*}
\varphi(x+hF(x)u) - \varphi(x) \leq \langle p, hF(x)u \rangle + \omega(|h|)|u|^{2}, \quad \forall\ p \in D^{+}_{F}\varphi(x). 
\end{equation*}
So, dividing by $h > 0$ we get
\begin{equation*}
\frac{\varphi(x+hF(x)u) - \varphi(x)}{h} \leq \langle p, F(x)u \rangle + \frac{\omega(|h|)}{h}|u|^{2}, \quad \forall\ p \in D^{+}_{F}\varphi(x). 
\end{equation*}
So, taking the minimum over $p \in D^{+}_{F}\varphi(x)$ on the right-hand side we have that 
\begin{equation*}
\frac{\varphi(x+hF(x)u) - \varphi(x)}{h} \leq \partial_{F}\varphi(x, u) + \frac{\omega(|h|)}{h}|u|^{2}.
\end{equation*}
Then, integrating the above inequality w.r.t. $\mu_{k}$ we obtain
\begin{multline*}
\int_{\T^{d} \times \R^{m}} \frac{\varphi(x+hF(x)u) - \varphi(x)}{h}\ \mu_{k}(dx, du) 
\\
\leq \int_{\T^{d} \times \R^{m}} \partial_{F}\varphi(x, u)\ \mu_{k}(dx, du) + \frac{\omega(|h|)}{h} \int_{\T^{d} \times \R^{m}} |u|^{2}\ \mu_{k}(dx, du)
\end{multline*}
Hence, as $k \uparrow \infty$ we get
\begin{equation*}
\int_{\T^{d} \times \R^{m}} \frac{\varphi(x+hF(x)u) - \varphi(x)}{h}\ \mu(dx, du) \leq \frac{\omega(|h|)}{h} \int_{\T^{d} \times \R^{m}} |u|^{2}\ \mu(dx, du).
\end{equation*}
and, as $h \to 0^{+}$ we have 
\begin{equation*}
\int_{\T^{d} \times \R^{m}} \partial_{F} \varphi(x, u)\ \mu(dx, du) \leq 0. 
\end{equation*}

Next, we proceed to show that 
\begin{equation*}
\int_{\T^{d} \times \R^{m}} \partial_{F} \varphi(x, u)\ \mu(dx, du) \geq 0. 
\end{equation*}
Again from the semiconcavity of $\varphi$ we have that 
\begin{equation*}
\varphi(x+hF(x)u) - \varphi(x) \geq \langle p, hF(x)u \rangle - \omega(|h|)|u|^{2}, \quad \forall\ p \in D^{+}_{F}\varphi(x+hF(x)u).
\end{equation*}
Appealing to the compactness of the state space and the upper-semicontinuity of the superdiffrential we have that for any $R \geq 0$, and any $\eps \geq 0$ there exists $h(R, \eps) \geq 0$ such that
\begin{equation*}
D^{+}_{F}\varphi(x+hF(x)u) \subset D^{+}_{F} \varphi(x) + \eps B_{1}, \quad \forall\ x \in \T^{d},\, \forall\  u \in B_{R}, \forall\ h \leq h(R, \eps). 
\end{equation*}
Hence, 
\begin{equation*}
\frac{\varphi(x+hF(x)u) - \varphi(x)}{h} \geq \partial_{F}\varphi(x, u) - \left(\eps |u| + \frac{\omega(|h|)}{h}|u|^{2} \right). 
\end{equation*}
Integrating w.r.t. $\mu_{k}$ over $\T^{d} \times B_{R}$ we get
\begin{align*}
& \int_{\T^{d} \times B_{R}} \frac{\varphi(x+hF(x)u) - \varphi(x)}{h}\ \mu_{k}(dx, du) 
\\
\geq\ & \int_{\T^{d} \times B_{R}} \partial_{F}\varphi(x, u)\ \mu_{k}(dx, du) - C\left(\eps + \frac{\omega(|h|)}{h} \right) \int_{\T^{d} \times B_{R}} |u|^{2}\ \mu_{k}(dx, du)
\end{align*}
So, 
\begin{multline*}
 \int_{\T^{d} \times B_{R}} \partial_{F}\varphi(x, u)\ \mu_{k}(dx, du) 
\\
=\  \int_{\T^{d} \times \R^{m}} \partial_{F}\varphi(x, u)\ \mu_{k}(dx, du) + \int_{\T^{d} \times B_{R}^{c}} \partial_{F}\varphi(x, u)\ \mu_{k}(dx, du)
\\
\geq\  \frac{C}{1+R}.
\end{multline*}
Thus, as $k \to \infty$ we get
\begin{equation*}
\int_{\T^{d} \times B_{R}} \frac{\varphi(x+hF(x)u) - \varphi(x)}{h}\ \mu(dx, du) \geq -C\kappa\left(\eps + \frac{\omega(|h|)}{h} + \frac{1}{R+1} \right).
\end{equation*}
Next, fix $\sigma \geq 0$ and let $R$ be such that $\frac{1}{R+1} \leq \sigma$, let $\eps \leq \sigma$. Then, 
\begin{equation*}
\int_{\T^{d} \times B_{R}} \frac{\varphi(x+hF(x)u) - \varphi(x)}{h}\ \mu(dx, du) \geq 2\sigma + o(|h|).
\end{equation*}
The conclusion follows as $h \to 0^{+}$ and since $\sigma$ is arbitrary. 
\qed

Next, we recall the construction of measures $\mu^{T}$ for $T > 0$ and, as before, we proceed to show that, up to subsequence, such measures weakly-$*$ converges to an $F$-strongly closed measure. 
Given $x_{0} \in \T^{d}$, for any $T > 0$ let the pair $(\gamma_{x_{0}}, u_{x_{0}}) \in \Gamma_{0,T}^{x_{0} \to}$ be optimal for \eqref{eq:evoValue}. Define the probability measure $\mu^{T}_{x_{0}}$ by 
\begin{equation}\label{eq:uniformmeasure_2}
\int_{\T^{d} \times \R^{m}}{\varphi(x,u)\ \mu^{T}_{x_{0}}(dx,du)}= \frac{1}{T}\int_{0}^{T}{\varphi(\gamma_{x_{0}}(t), u_{x_{0}}(t))\ dt}, \quad \forall\ \varphi \in C_{b}(\T^{d} \times \R^{m}).
\end{equation}
Then, we have the following.  
 
\begin{proposition}
The sequence $\{ \mu^{T}_{x_{0}}\}_{T>0}$ is tight and there exists a sequence $T_{n} \to \infty$ such that $\mu^{T_{n}}_{x_{0}}$ weakly-$^*$ converges to an $F$-strongly closed measure $\mu^{\infty}_{x_{0}}$.  
\end{proposition}
\proof

Proceeding as in \Cref{prop:nonempty} one immediately gets that there exists a subsequence $\{T_n\}_{n \in \N}$ and a probability measure $\mu^{\infty}_{x_{0}}$ such that $\mu^{T_n}_{x_{0}}$ weakly-$*$ converges to $\mu^{\infty}_{x_{0}}$. 
Let us show that $\mu^{\infty}_{x_0}$ is $F$-strongly closed, i.e., we have to show that
\begin{equation*}
\int_{\T^{d} \times \R^{m}} \partial_{F} \varphi(x, u)\ \mu^{\infty}_{x_0}(dx, du) = 0, \quad \forall\ \varphi \in \SC(\T^{d}). 
\end{equation*}
Fix $\varphi \in \SC(\T^{d})$. Then, 
\begin{align*}
\begin{split}
& \int_{\T^{d} \times \R^{m}} \partial_{F} \varphi(x, u)\ \mu^{T_{k}}_{x_0}(dx, du) 
\\
=\ & \frac{1}{T_{k}}\int_{0}^{T_k} \lim_{h \to 0^{+}} \frac{\varphi(\gamma(t) + hF(\gamma(t))u(t)) - \varphi(\gamma(t))}{h}\ dt
\\
=\ &  \lim_{h \to 0^{+}} \frac{1}{T_{k}}\int_{0}^{T_k} \left[\frac{\varphi(\gamma(t+h)) - \varphi(\gamma(t))}{h} +\frac{\varphi(\gamma(t) + hF(\gamma(t))u(t)) - \varphi(\gamma(t+h))}{h}  \right]\ dt.
\end{split}
\end{align*}
Recalling that since $\varphi \in \SC(\T^{d})$ then $\varphi \in W^{1,\infty}(\T^{d})$, we immediately get
\begin{align*}
\begin{split}
& \left|\frac{\varphi(\gamma(t) + hF(\gamma(t))u(t)) - \varphi(\gamma(t+h))}{h}\right| 
\\
\leq\ & \frac{\|D\varphi\|_{\infty} |\gamma(t)+hF(\gamma(t))u(t) - \gamma(t+h)|}{h} \leq \|D\varphi\|_{\infty} o(|h|).
\end{split}
\end{align*}
On the other hand, 
\begin{align*}
 & \frac{1}{T_{k}}\int_{0}^{T_k} \frac{\varphi(\gamma(t+h)) - \varphi(\gamma(t))}{h} = \frac{1}{T_k} \left[\int_{h}^{T_{k}+h} \frac{\varphi(\gamma(t))}{h}\ dt - \int_{0}^{T_k} \frac{\varphi(\gamma(t)}{h}\ dt \right]
 \\
 =\ & \frac{1}{T_k} \int_{0}^{h} \frac{\varphi(\gamma(t))}{h}\ dt - \frac{1}{T_k} \int_{T_k}^{T_{k}+h} \frac{\varphi(\gamma(t))}{h}\ dt  \leq \frac{2}{T_{k}}\| \varphi\|_{\infty}.
\end{align*}
Thus, as $k \uparrow \infty$ we conclude
\begin{equation*}
\int_{\T^{d} \times \R^{m}} \partial_{F} \varphi(x, u)\ \mu^{\infty}_{x_0}(dx, du) = 0, \quad \forall\ \varphi \in \SC(\T^{d}). \eqno\square
\end{equation*}

\begin{theorem}\label{thm:main0}
We have that
\begin{equation*}
\mane  = \min_{\mu \in \C^{S}_{F}} \int_{\T^{d} \times \R^{m}}{L(x,u)\ \mu(dx,du)}. 
\end{equation*}
\end{theorem}
\proof
First, by definition 
\begin{equation*}
\min_{\mu \in \C^{S}_{F}} \int_{\T^{d} \times \R^{m}}{L(x,u)\ \mu(dx,du)} \geq \mane.
\end{equation*}
On the other hand, 
 \begin{align*}
  \mane =  \lim_{T \to \infty}\int_{\T^{d} \times \R^{m}}{L(x,u)\ \mu^{T}_{0}(dx,du)} \geq \int_{\T^{d} \times \R^{m}}{L(x,u)\ \mu^{\infty}_{0}(dx,du)}
 \end{align*}
 and the conclusion follows taking the infimum over $\C^{S}_{F}$. 
\qed

\section{Mather set}\label{Mather_Set}

In this section we introduce the Mather set associated with a control system with nonholonomic constraints. We show that such a set is included into the projected Aubry set and, as a consequence, a fixed-point of the Lax-Oleinik semigroup is horizontally differentiable on the projected Mather set. Before doing this, we prove the existence of a strict critical subsolution (\Cref{strict_subsolution} below) and, then, the existence of a semiconcave strict critical subsolution which is, consequently, globally Lipschitz continuous w.r.t. the Euclidean distance. 
		
		\subsection{Lax-Oleinik semigroup and viscosity solutions}

		In this section we investigate the relation between the Lax-Oleinik semigroup and viscosity solutions to 
		\begin{equation*}
		\mane + H(x, D\chi(x)) = 0, \quad x \in \T^d. 
		\end{equation*}

		Before proving the main result of this section we need a technical lemma on the propagation of regularity along optimal pairs via the inclusion of the dual arc in the subdifferential of the value function. The proof of the following lemma can be obtained by adapting the arguments in the proof of \cite[Theorem 7.3.4]{bib:SC} to control-affine state equations. 
		
		\begin{lemma}\label{lem:propagation}
		Let $\varphi \in C(\T^d)$, let $(t, x) \in [0, \infty) \times \T^d$ and let $(\gamma, u) \in \Gamma_{0,t}^{\to x}$ be optimal for 
		\begin{equation*}
		v_{\varphi}(t, x) = T_t \varphi(x) = \inf_{(\gamma, u) \in \Gamma_{0,t}^{\to x}} \left\{\varphi(\gamma(0)) + \int_{0}^{t} L(\gamma(s), u(s))\ ds \right\}.
		\end{equation*}
		Let $p : [0, t] \to \R^d$ be the dual arc associated with the above problem and suppose that $p(t) \in D_{x}^{-} v_{\varphi}(t, x)$. Then, 
		\begin{equation*}
		p(s) \in D_{x}^{-} v_{\varphi}(s, \gamma(s)), \quad \forall\ s \in [0, t]. 
		\end{equation*}
		\end{lemma}

		\begin{remark}\label{continuity}\em
		Observe that \Cref{lem:propagation} still holds true replacing $D^{-}v_{\varphi}$ with $D^{*}v_{\varphi}$, where $D^*v_{\varphi}$ denotes the set of all limiting horizontal gradients of $v_{\varphi}$. This fact follows arguing as in \cite[Theorem 7.3.10]{bib:SC} and \cite{bib:CM}. 
		\end{remark}

		\begin{theorem}\label{global}
		A continuous function $v: \T^d \to \R$ is a viscosity solution to 
		\[
		\mane + H(x, Dv(x)) = 0 \quad (x \in \T^d)
		\]
		if and only if $v(x) = T_t v(x) - \mane t$ for any $t \geq 0$. 
		\end{theorem}
		\proof 
		If $v(x) = T_t v(x) - \mane t$, then the statement follows by the dynamic programming principle. 
		
		Let us prove the converse. To do so, let $v$ be a viscosity solution to 
		\begin{equation*}
		\mane + H(x, Dv(x)) = 0, \quad x \in \T^d
		\end{equation*}
		and define the function 
		\begin{equation*}
		\widetilde v(t, x) = T_t v(x), \quad \forall\, t \geq 0, \,\, \forall\, x \in \T^d. 
		\end{equation*}
		So, the proof is reduced to showing that $\widetilde v(t, x) = v(x) + \mane t$. From \Cref{lem:semiconcavity} we deduce that $\widetilde v(t, x)$ is semiconcave and, thus, Lipschitz continuous w.r.t. the Euclidean distance. This implies that it is enough to show that $\partial_t \widetilde v(t, x) = \mane $ at any point $(t, x)$ where $\widetilde v(t, x)$ admits a derivative. By the dynamic programming principle it follows that $\widetilde v$ satisfies 
		\begin{equation*}
		\partial_t \widetilde v(t, x)  + H(x, D\widetilde v(t, x)) = 0, \quad (t, x) \in [0, \infty) \times \T^d
		\end{equation*}
		in the viscosity sense. Hence, we have to show that 
		\begin{equation*}
		\mane + H(x, D\widetilde v(t, x)) = 0, \quad (t, x) \in [0, \infty) \times \T^d. 
		\end{equation*}
		Since $\widetilde v(t, x) \prec L - \mane$, we have that
		\begin{equation*}
		\mane +  H(x, D\widetilde v(t, x)) \leq 0. 
		\end{equation*}
		Next, we proceed to show that $\mane +  H(x, D\widetilde v(t, x)) \geq 0$ holds. Let $(t, x) \in (0, \infty) \times \T^d$ be a point of differentiability of $\widetilde v$ and let $(\gamma, u) \in \Gamma_{0,t}^{ \to x}$ satisfy
		\begin{equation*}
		T_t v(x) = v(\gamma(0)) + \int_{0}^{t} L(\gamma(s), u(s))\ ds. 
		\end{equation*}
		By the maximum principle and the properties of the Legendre transform we obtain 
		\begin{equation*}
		\partial_u L(x, u(t)) \in D^{+}_{F} \widetilde v(t, x)
		\end{equation*}
		and since $\widetilde v$ is differentiable at $(t, x)$ we get $D_F \widetilde v(t, x) = \partial_u L(x, u(t))$. Using the fact that
		\begin{equation*}
		H(\gamma(s), \partial_u L(\gamma(s), u(s)))
		\end{equation*}
		is constant we are reduced to show 
		\begin{equation*}
		\mane + H(\gamma(0), \partial_u L(\gamma(0), u(0))) \geq 0. 
		\end{equation*}
		The above inequality is a consequence of \Cref{lem:propagation}, since 
		\begin{equation*}
		\partial_u L(\gamma(0), u(0)) \in D^{-}v(\gamma(0))
		\end{equation*}
		and $v$ is a viscosity solution to $\mane + H(x, Dv(x)) = 0$. \qed

		\subsection{Strict critical subsolution}

		\begin{definition}[{\bf Strict critical subsolution}]\label{strict_subsolution}
		A subsolution $\chi: \T^{d} \to \R$ to the critical equation is said to be strict at $x \in \T^{d}$ if there exists an open subset $U \subset \T^{d}$ and $c > \mane$ such that $x \in U$, and $\chi_U=\chi|U$ is a subsolution of $c + H(x, D \chi_U)=0$.  
		\end{definition}

		
		Before proving the existence of a strict critical subsolution to the critical equation we provide a preliminary result on the, so called, Ma\~n\'e potential function. We omit the proof which is similar to the one of \cite[Proposition 4.2]{bib:FAS}.

		\begin{proposition}
		Let $x \in \T^d$ and let $\phi_x$ be function defined by 
		\begin{equation}\label{potential}
	\phi_x(y) = \inf_{t \geq 0} [A_t(x, y) - \mane t].
	\end{equation}
 Then, $\phi_x(\cdot)$ is a viscosity subsolution to the critical equation on $\T^d$. Moreover, $\phi_x(\cdot)$ is a viscosity solution to the critical equation on $\T^d$ if and only if $x \in \A$.  
		\end{proposition}

		\begin{remark}\label{rem1}\em
		Since, for any $x \in \T^d$, $\phi_{x}$ is a subsolution to the critical equation 
		\begin{equation*}
		\mane + H(y, D\phi_{x}(y)) = 0
		\end{equation*} 
		on $\T^{d}$, by appealing to \Cref{lem:equilip} we conclude that $\phi_{x}(\cdot)$ is locally Lipschitz continuous uniformly w.r.t. $x \in \T^d$. 
		\end{remark}
		
		Next, we prove the existence of a strict critical subsolution to the equation 
		\begin{equation*}
		\mane + H(x, D\chi(x)) = 0, \quad x \in \T^d. 
		\end{equation*}
		
		\begin{proposition}\label{strict}
		There exists a global critical subsolution which is strict at each point of $\T^{d} \backslash \A$. 
		\end{proposition}
		\proof 
		We first show that, given $x \in \T^{d} \backslash \A$, we can construct a global critical subsolution $\chi_x: \T^{d} \to \R$ which is strict on an open set $U_x$ containing $x$.

		Let $\phi^{x}: \T^{d} \to \R$ be defined as in \eqref{potential}. Appealing to \Cref{rem1} we have that such a function is a global subsolution to the critical equation and a solution to the critical equation on $\T^{d} \backslash \{x_{0}\}$, that is, the viscosity supersolution  condition must be violated at $x$. So, there exists an horizontally differentiable function $\theta: V \to \R$ defined on a open neighborhood $V$ of $x$ such that 
		\begin{equation*}
		\theta(y) < \phi^{x}(y) \,\, \text{for}\,\, y \in V, \quad \theta(x)=\phi^{x}(x), \quad \text{and} \quad \mane+ H(x, D\theta(x)) <0 .
		\end{equation*}
		Hence, we can find a constant $c_x > \mane $ and a open neighbourhood $\mathcal{W}$ of $x$ whose closure is compact and contained in $V$ such that 
		\[
		c_x + H(y, D\theta(y)) < 0 \quad (y \in \overline\W).
		\]

		Next, let $\eps > 0$ be such that $\phi^{x} (y) > \theta (y) + \eps $ for $y \in \overline\W \backslash \W$ and define 
		\begin{equation*}
		v_{x}(y) = \max \{\phi^{x}(y), \theta(y) + \eps\}, \quad y \in \overline\W.
		\end{equation*}
		Note that, by the choice of $\eps$ we have that $\phi^{x}(y) = \phi^x (y)$ on $\overline\W \backslash \W$ and so the function $v_x$ can be continuously extended on $\T^d$. Moreover, observe that $\theta(x) + \eps > \phi^{x}(x) = \theta(x)$ which implies that there exists an open neighbourhood $U_x$ of $x$ such that  $\phi^x (y) = \theta(y) + \eps$, and for any $y \in U_x$, we have that $c_x + H(y, D \phi^x(y)) = c_x + H(y, D\theta(y)) < 0$. Thus, we deduce that $\phi^{x}_{|U_{x}}$ is dominated by $L-c_x$ on $U_{x}$.

By definition, $\phi^{x}(x)=0$. So we have that
		\begin{equation}\label{equi-bound}
		\sup\{|\phi^{x} (y)| : x \in \T^d \backslash \A,\,\, y \in K \} < \infty
		\end{equation}
		for each compact subset $K$ of $\T^d$. 
		Since the set $\T^d \backslash \A$ is covered by the open sets $U_x$ with $x \in \T^d \backslash \A$, we can find a countable subcover  $\{ U_{x_{n}}\}_{n \in \N}$. Define the function 
		\begin{equation*}
		v(y) = \sum_{k=1}^{\infty} \frac{1}{2^k} \phi_{x_{k+1}}(y).
		\end{equation*}
		From \eqref{equi-bound} we have that $\sum_{k=1}^{\infty} \frac{1}{2^k} \phi_{x_{k+1}}(y)$ uniformly converges on all compact subsets of $\T^d$. Moreover, we have that $v$ is an infinite convex combination of $\phi^{x_{k}}$. Hence, we get that $v$ is dominated by $L-\mane$ since each $\phi^{x_{k}} \prec L-\mane$ and $v_{| U_{x_{k}}}$ is dominated by $L-c$ with 
		\begin{equation*}
		c:=\frac{c_{x_k}}{2^{k+1}} + \sum_{m \not=k} \frac{\mane}{2^{m+1}}< \mane.
		\end{equation*}
	This yields the fact that $v$ is a strict subsolution of the critical equation. \qed

\subsection{Analysis of the Mather set}
In order to prove the main result of this section, namely \Cref{inclusion}, we need to use a strict critical subsolution as a test function for closed measures. For this, we have to construct a semiconcave subsolution, which is strict at each point not in the projected Aubry set. This will be done by showing that, given any strict critical subsolution, for any $t > 0$, the function $T_t v: \T^d \to \R$ is also a strict critical subsolution.

Before doing so, let us recall that the Lax-Oleinik semigroup is defined on the set of dominated functions $\D_H$ by the formula
\begin{equation*}
T_{t} \varphi(x) = \inf_{(\gamma, u) \in \Gamma_{0, t}^{\to x}} \left\{\varphi(\gamma(0)) + \int_{0}^{t} L(\gamma(s), u(s))\ ds \right\} \quad (\varphi \in \D_H, \; x \in \T^d).
\end{equation*}
Let us introduce the set 
\begin{equation*}
\I(\varphi) = \{x \in \T^{d} : \exists\ (\gamma, u) \in \Gamma_{-\infty, 0}^{\to x} \cap \Gamma_{0, \infty}^{x \to} \,\, \text{s.t.}\,\, \gamma \,\,\text{is calibrated for}\,\, \varphi\}. 
\end{equation*}

The proof of the following Lemma, as well as the one of the next Theorem, can be recovered arguing as in \cite{bib:FAS}, Proposition 6.2 and Lemma 6.3 respectively. 

\begin{lemma}\label{lem1}
The following facts hold true. 
\begin{itemize}
\item[($i$)] If $v$ is a subsolution of the critical equation then, for each $t > 0$, the function $T_{t}v$ is also a subsolution of the critical equation. Moreover,
\[
\I(T_{t}v)=\I(v), \quad \textrm{and} \quad v = T_{t}v -\mane t \quad \textrm{ on }\;\; \I(T_{t}v)=\I(v)\; \forall\; t > 0.
\]
\item[($ii$)] If $t > 0$, $x \in \T^{d} \backslash \I(v)$, and $\gamma: [0,t] \to \T^{d}$ are such that $x=\gamma(t)$, and 
\[
T_{t}v(x)= v(\gamma(0)) + \int_{0}^{t} L(\gamma(s), u(s))\ ds,
\]
then, $\gamma\left([0,t]\right) \subset \T^{d} \backslash \I(v)$. 
\end{itemize}
\end{lemma}

\begin{theorem}\label{Tstrict}
 Let $v$ be a subsolution of the critical equation which is strict at each point of $\T^{d} \backslash \I(v)$. Then $T_{t}v$ is also a subsolution of the critical equation, strict at each point of $\T^{d} \backslash \I(v)$. 
\end{theorem}

Let $\pi_{1}: \T^{d} \times \R^{m} \to \T^{d}$ be the projection onto the first factor, that is, $\pi_{1}(x,u)=x$.
\begin{definition}[{\bf Mather measures and Mather set}]
		We say that $\mu \in \C_{F}^{S}$ is a Mather measure for $L$ if  
		\begin{align*}
		\mane = \inf_{\nu \in \C_{F}^{S}} \int_{\R^{d} \times \R^{m}}{L(x,u)\ \nu(dx,du)}=\int_{\R^{d} \times \R^{m}}{L(x,u)\ \mu(dx,du)}	
		\end{align*}
		and $\C^{*}_{F}(L)$ denotes the set of all Mather measures. We define the Mather set as 
		\begin{align*}
		\widetilde{\M}=\overline{\bigcup_{\mu \in \C^{*}_{F}(L)} \supp(\mu)} \subset \T^{d} \times \R^{m}.	
		\end{align*}
		and we call $\M = \pi_{1}\big(\widetilde{\M} \big)$ the projected Mather set. 
		\end{definition}

We can now prove the main result of this section stating the inclusion of the projected Mather set into the projected Aubry set.

\begin{theorem}\label{inclusion}
The following holds:
\begin{equation*}
\M \subset \A.
\end{equation*}  
\end{theorem}
\proof 

Let $x_0 \in \T^d \backslash \A$, let $V$ be an open neighborhood of $x$ with $V \subset \T^d \backslash \A$ and let $\varphi$ be a semiconcave strict subsolution at $x_0$ to the critical equation with value $c_0 > \mane$. Let $x \in V$ and let $p \in D^{+}_{F} \varphi(x)$. Then, by the Legendre transform we have that 
\begin{align*}
\langle p, u\rangle_{\R^m} \leq L(x,u) + L^*(x, p) < L(x,u) - c_0, \quad \forall\ (x, u) \in V \times \R^m.
\end{align*}
So, taking the minimum over $p \in D^{+}_{F} \varphi(x)$ we get
\begin{equation*}
\partial_{F} \varphi (x, u) \leq  L(x,u) - c_0, \quad \forall\ (x, u) \in V \times \R^m. 
\end{equation*}
Fix a Mather measure $\mu \in \C^{*}_{F}(L)$. Then, we have that 
\begin{align*}
0 = \ & \int_{\T^d \times \R^m} \partial_{F} \varphi (x, u)\ \mu(dx,du)
\\
=\ & \int_{V \times \R^m} \partial_{F} \varphi (x, u)\ \mu(dx,du) + \int_{(\T^d \backslash V) \times \R^m} \partial_{F} \varphi (x, u)\ \mu(dx,du)
\\
\leq\ & \int_{V \times \R^m} (L(x,u) - c_0)\ \mu(dx,du) + \int_{(\T^d \backslash V) \times \R^m} (L(x,u) - \mane)\ \mu(dx,du)
\\
=\ & \int_{V \times \R^m} (\mane - c_0)\ \mu(dx,du) + \int_{\T^d \times \R^m} (L(x,u) - \mane)\ \mu(dx,du). 
\end{align*}
Hence, recalling that $c_0 > \mane$ and $\int_{\T^d \times \R^m} (L(x,u) - \mane)\ \mu(dx,du) = 0$ we conclude that 
\begin{equation*}
\int_{\T^d \times \R^m} \partial_{F} \varphi (x, u)\ \mu(dx,du) < 0
\end{equation*}
which implies the conclusion. \qed

As a consequence of the above result, we obtain the horizontal differentiability of solutions to the critical equation in $\M$ and the representation of $\M$ as the graph of a continuous vector fields (in the spirit of Mather graph Theorem).

\begin{corollary}\label{matherlemma}
Let $\mu \in \C^{*}_{F}(L)$ be a Mather measure and let $\chi$ be a fixed-point of $T_t$. Then the following holds:
\begin{itemize}
\item[($i$)] $\chi$ is differentiable at any $x \in \supp(\pi_{1} \sharp \mu)$ along the range of $F(x)$ and the horizontal gradient $D_{F}\chi(x)$ satisfies
\begin{equation*}
 D_{F}\chi(x)= D_{u}L(x,u), \quad \forall\ (x,u) \in \supp(\mu);
\end{equation*}	
\item[($ii$)] $\chi$ is a solution to 
\begin{equation*}
\mane + L^{*}(x, D_{F}\chi(x))=0, \quad \forall\ x  \in \supp(\pi_{1} \sharp \mu)
\end{equation*} 
in the classical sense;
\item[($iii$)] the Mather set is the graph of a continuous vector field, that is, 
\begin{equation*}
\widetilde\M=\{(x,D_{p}L^{*}(x,D_{F}\chi(x)): x \in \M\}.
\end{equation*}
\end{itemize}
\end{corollary}
\proof
	Let $\chi$ be a fixed-point of the Lax-Oleinik semigroup. Since $\M \subset \A$ and $\chi$ is horizontally differentiable on $\A$, we deduce that the same holds on $\M$. Moreover, we have that
\begin{equation}\label{eq:inequality}
\langle D_{F}\chi(x), u \rangle \leq L(x,u) + L^{*}(x, D_{F}\chi(x)) \leq L(x,u) - \mane, \,\, (x,u) \in \M \times \R^{m}.	
\end{equation}
Hence, integrating the above inequality w.r.t. $\mu$ we get
 \begin{equation*}
 0 \leq \int_{\R^{d} \times \R^{m}}{\big(L(x,u) + L^{*}(x, D_{F}\chi(x)) \big)\ \mu(dx,du)} \leq 0.
 \end{equation*}
 So,
 \begin{equation}\label{eq:0}
\int_{\R^{d} \times \R^{m}}{\big(L(x,u) + L^{*}(x, D_{F}\chi(x)) \big)\ \mu(dx,du)} = 0.
 \end{equation}
From \eqref{eq:inequality} we know that 
\begin{equation}\label{eq:00}
0 \leq L(x,u) + L^{*}(x, D_{F}\chi(x)) - \langle D_{F}\chi(x), u \rangle
\end{equation}
Therefore, combining \eqref{eq:0} and \eqref{eq:00} we deduce that 
\begin{equation*}
0 \leq \int_{\R^{d} \times \R^{m}}{\big(L(x,u) + L^{*}(x, D_{F}\chi(x)) - \langle D_{F}\chi(x), u \rangle \big)\ \mu(dx, du)} = 0
\end{equation*}
and so
\begin{equation}\label{eq:longequality}
	\langle D_{F}\chi(x), u \rangle = L(x,u) + L^{*}(x, D_{F}\chi(x)), \quad (x,u) \in \supp(\mu). 
\end{equation}
Similarly, again from \eqref{eq:inequality} and \eqref{eq:0} we obtain 
\begin{equation*}
L(x,u) + L^{*}(x, D_{F}\chi(x)) = L(x,u) - \mane, \quad (x,u) \in \supp(\mu). 
\end{equation*}
Thus, 
\begin{equation*}
\mane + L^{*}(x, D_{F}\chi(x)) = 0, \quad \forall\ x \in \supp(\pi_{1} \sharp \mu)	
\end{equation*}
and, by \eqref{eq:longequality} and the properties of the Legendre transform (\cite[Appendix A.2]{bib:SC}), we get 
\begin{equation*}
	  D_{F}\chi(x)=D_{u}L(x,u), \quad \forall \,\, (x,u) \in \supp(\mu)
\end{equation*}
which proves ($i$) and ($ii$). Moreover, still from \eqref{eq:longequality} and the properties of the Legendre transform we deduce that
\begin{equation*}
u=D_{p}L^{*}(x, D_{F}\chi(x)), \quad \forall\ (x,u) \in \supp(\mu).	
\end{equation*}
So,
\begin{equation*}
\widetilde\M=\{(x, D_{p}L^{*}(x, D_{F}\chi(x)): x \in \M\} 	
\end{equation*}
which proves ($iii$) and completes the proof. \qed

We conclude this section with an example. 

\begin{example}\em\label{Grushin1}
Let us consider the Grushin dynamic on $\T^2$ determined by
\begin{equation*}
\mathcal{G}(x)=
\begin{bmatrix}
1 & 0 \\ 0 & x_1
\end{bmatrix}
\end{equation*}
with controls $u = (u_1, u_2) \in \R^2$ and a Lagrangian of the form
\begin{equation*}
L(x, u) = \frac{1}{2} |u - V(x)|^2 + G(x)
\end{equation*}
for a given $V \in C^2(\T^2; \R^2)$ and $G \in C^2(\T^2; \R)$. The state equation associated with  $\mathcal{G}$ is
\begin{equation*}
\begin{cases}
\dot x_1(t) = u_1(t)
\\
\dot x_2(t) = x_1(t) u_2(t). 
\end{cases}
\end{equation*}
We recall that a probability measure $\mu$ on $\T^2 \times \R^2$ is $\mathcal{G}$-closed if 
\begin{equation*}
	\int_{\T^{2} \times \R^{2}}{\langle D\varphi(x), \mathcal{G}(x)u \rangle\ \mu(dx,du)}=0, \quad \forall\ \varphi \in C^{1}(\T^{2}). 	
	\end{equation*}
	A simple verification shows that the following measures are $\mathcal{G}$-closed: 
\begin{equation*}
\mu_1=\nu \otimes \delta_{(0, 0)}, \quad \textrm{where $\nu$ is any probability measure on $\T^{2}$}
\end{equation*}
and
\begin{equation*}
\mu_2=\nu^1_{(0, \T)} \otimes \nu^2_{(0, \R)},
\end{equation*}
where $\nu^1_{(0,\T)} = \delta_{\{0\}} \otimes \nu_1$, $\nu^2_{(0, \R)} = \delta_{\{0\}} \otimes \nu_2$ for any $\nu_1$ and $\nu_2$, probability measures on $\T$ and $\R$ respectively. 
So, recalling also that 
\begin{equation*}
\mane = \inf_{\mu \in \C_{F}} \int_{\T^2 \times \R^2} \left(\frac{1}{2}|u - V(x)|^2 + G(x) \right)\;\mu(dxdu)
\end{equation*}
we obtain
\begin{multline*}
\mane \leq \inf \left\{\int_{\T^2 \times \R^2} \left(\frac{1}{2}|u - V(x)|^2 + G(x) \right)\;\mu(dxdu) : \mu = \nu \otimes \delta_{(0,0)}, \; \nu \in \PP(\T^2)  \right\}
\\
= \min_{y \in \T^2} \left(\frac{1}{2}|V(y)|^2 + G(y) \right)
\end{multline*}
and
\begin{align*}
\mane & \leq \inf \Big\{\int_{\T^2 \times \R^2} \left(\frac{1}{2}|u - V(x)|^2 + G(x) \right)\;\mu(dxdu) : \\ &\qquad\qquad\qquad\qquad \mu = (\delta_{\{0\}} \otimes \nu_1) \otimes (\delta_{\{0\}} \otimes \nu_2),\; \nu_1 \in \PP(\T), \; \nu_2 \in \PP(\R)  \Big\}
\\
& =\ \min_{z \in \T} \left(\frac{1}{2}|V_1(0, z)|^2 + G(0, z)\right)
\end{align*}
where we set $V(x)=(V_1(x), V_2(x))$ for $x \in \T^2$. Therefore, 
\begin{equation}\label{Gmane}
\mane \leq \min\left\{\min_{y \in \T^2} \left(\frac{1}{2}|V(y)|^2 + G(y) \right), \min_{z \in \T} \left(\frac{1}{2}|V_1(0, z)|^2 + G(0, z)\right)\right\}. 
\end{equation}

In some special cases, it is easy to show that \eqref{Gmane} becomes an equality. For instance, if there exists $(x^*, y^*) \in \T^2$ such that $V(x^*, y^*) = 0$ and $G(x^*, y^*) = \displaystyle{\min_{(x, y) \in \T^2}} G(x, y)$, then  
\begin{equation*}
\mane = \min_{(x, y) \in \T^2} G(x, y)
\end{equation*}
and $\mu_1 = \delta_{(x^*, y^*)} \otimes \delta_{(0, 0)}$ is the Mather measure. 
\end{example}

\medskip
\small{
\noindent{\bf Declarations.}

{\bf Ethical approval.} Not applicable.

{\bf Competing interests.} The authors report there are no competing interests to declare.

{\bf Authors contributions.} All authors contributed to the research, editing, and review of the manuscript.

{\bf Funding.} The authors were partially supported by Istituto Nazionale di Alta Matematica, INdAM-GNAMPA project 2022 and INdAM-GNAMPA project 2023, by the MIUR Excellence Department Project MatMod@TOV awarded to the Department of Mathematics, University of Rome Tor Vergata, CUP E83C23000330006, and by the King Abdullah University of Science and Technology (KAUST) project CRG2021-4674 "Mean-Field Games: models, theory and computational aspects". P. Cannarsa was partially supported by PRIN 2022 PNRR "Some mathematical approaches to climate change and its impacts", CUP E53D23017910001.

{\bf Availability of data and materials.} No data associated in the manuscript. 
}


\end{document}